\RequirePackage{fix-cm}
\RequirePackage{amsmath}
\documentclass[smallextended]{svjour3}       
\smartqed  
\usepackage{graphicx}
\usepackage{amsfonts}
\usepackage{amssymb}
\usepackage{latexsym}
\usepackage{mathtools}
\mathtoolsset{showonlyrefs}
\usepackage{listings}
\usepackage{graphicx}
\usepackage{color}
\usepackage{xcolor}
\usepackage{fullpage}
\usepackage{hyperref}

\newcommand{\Input}{\mathcal{V}}

\def\C{\mathbb{C}}
\def\R{\mathbb{R}}
\def\F{\mathbb{F}}

\def\SS{\mathcal{S}}

\def\rank{\mathrm{rank}}

\newcommand{\lowbiglquote}[1][18]{%
   \setbox0=\hbox{\fontsize{#1}{0}\selectfont``}%
   \setlength{\dimen0}{\ht0 - \heightof{A}}%
   \noindent\llap{\smash{\lower\dimen0\box0 }}}
 
\newcommand{\lowbigrquote}[1][18]{%
   \setbox0=\hbox{\fontsize{#1}{0}\selectfont''}%
   \setlength{\dimen0}{\ht0 - \heightof{A}}%
   \unskip\rlap{\smash{\lower\dimen0\box0 }}}

\def\Prob{\mathbb{P}}

\newcommand{\Expect}{\operatorname{\mathbb{E}}}
\newcommand{\proofof}[1]{\par{\it Proof of {#1}}. \ignorespaces}
\newcommand{\diff}[1]{\mathrm{d}{#1}}
\newcommand{\normal}{\mathcal{N}}
\newcommand{\uniform}{\mathcal{U}}
\newcommand{\ginibre}{\mathcal{G}}
\newcommand{\CUE}{\mathrm{CUE}}
\newcommand{\CRE}{\mathrm{CRE}}
\newcommand{\Beta}{\mathrm{B}}

\def\endproof{\vbox{\hrule height0.6pt\hbox{%
   \vrule height1.3ex width0.6pt\hskip0.8ex
   \vrule width0.6pt}\hrule height0.6pt
  }}

\DeclareMathOperator{\diag}{diag}

\spnewtheorem{theorem}{Theorem}[section]{\bfseries}{\itshape}

\spnewtheorem{lemma}[theorem]{Lemma}{\bfseries}{\itshape}
\spnewtheorem{definition}[theorem]{Definition}{\bfseries}{\itshape}
\spnewtheorem{remark}[theorem]{Remark}{\bfseries}{\upshape}
\spnewtheorem{example}[theorem]{Example}{\bfseries}{\upshape}
\spnewtheorem{corollary}[theorem]{Corollary}{\bfseries}{\itshape}
\spnewtheorem{proposition}[theorem]{Proposition}{\bfseries}{\itshape}
\spnewtheorem{algorithm}[theorem]{Algorithm}{\bfseries}{\upshape}

\numberwithin{theorem}{section}

\begin{document}

\title{Wilkinson's bus: Weak condition numbers,\\ with an application to singular polynomial eigenproblems}

\titlerunning{Weak condition numbers, with applications}        

\author{Martin Lotz         \and
        Vanni Noferini 
}

\institute{M. Lotz \at
              Mathematics Institute, The University of Warwick.
              \email{martin.lotz@warwick.ac.uk}           
           \and
           V. Noferini \at
              Department of Mathematics and Systems Analysis, Aalto University.
              \email{vanni.noferini@aalto.fi}  
}

\date{Received: date / Accepted: date}

\maketitle

\begin{abstract}
We propose a new approach to the theory of conditioning for numerical analysis problems for which both classical and stochastic perturbation theory fail to predict the observed accuracy of computed solutions. To motivate our ideas, we present examples of problems that are discontinuous at a given input and even have infinite stochastic condition number, but where the solution is still computed to machine precision without relying on structured algorithms. 
Stimulated by the failure of classical and stochastic perturbation theory in capturing such phenomena, we define and analyse a weak worst-case and a weak stochastic condition number. This new theory is a more powerful predictor of the accuracy of computations than existing tools, especially 
when the worst-case and the expected sensitivity of a problem to perturbations of the input is not finite. We apply our analysis to the computation of simple eigenvalues of matrix polynomials, including the more difficult case of singular matrix polynomials.  
In addition, we show how the weak condition numbers can be estimated in practice.

\end{abstract}

\section{Introduction}
The \emph{condition number} of a computational problem 
measures the sensitivity of an output with respect to perturbations in the input. 
If the input-output relationship can be described by a differentiable function $f$ near the input, 
then the condition number is the norm of the derivative of $f$. In the case of solving systems of linear equations, the idea of conditioning dates back at least to the work of von Neumann and Goldstine~\cite{Neumann1947} and Turing~\cite{Turing1948}, who coined the term.
For an algorithm computing $f$ in finite precision arithmetic, the importance of the condition number $\kappa$ stems from the ``rule of thumb'' popularized by N.~J.~Higham~\cite[\S 1.6]{Higham1996},
$$\mathrm{forward} \ \mathrm{error} \lesssim \kappa \cdot (\mathrm{backward} \ \mathrm{error}).$$
The backward error is small if the algorithm computes the exact value of $f$ at a {\em nearby} input, and a small condition number would certify that this is enough to get a small overall error.
Higham's rule of thumb comes from a first order expansion, and in practice it often holds as an approximate equality and is valuable for practitioners
who wish to predict the accuracy of numerical computations. Suppose that a solution is computed with, say, a backward error equal to $10^{-16}$. If $\kappa=10^2$ then one would trust the computed value to have (at least) $14$ meaningful decimal digits.

The condition number can formally still be defined when $f$ is not differentiable, though it may not be finite. 
If $f$ is not locally Lipschitz continuous at an input, then the condition number is $+ \infty$; a situation clearly beyond the applicability of Higham's rule. 
Inputs at which the function $f$ is not continuous are usually referred to as {\em ill-posed}. Based on the worst-case sensitivity one would usually only expect a handful of correct digits when evaluating a function at such an input, and quite possibly 
none.\footnote{The number of accurate digits that can be expected when the problem is continuous but not locally Lipschitz continuous requires a careful discussion. It depends on the unit roundoff $u$, on the exact nature of the pathology of $f$, and on $D$. For example, computing the eigenvalues of a matrix similar to an $n\times n$ Jordan block for $n>1$ is H\"older continuous with exponent $1/n$ but not Lipschitz continuous. 
Usually this translates into expecting only about $\sqrt[n]{u}$ accuracy, up to constants, when working in finite precision arithmetic. For a more complete discussion, see \cite{GW1976}, where pathological examples of derogatory matrices are constructed, whose eigenvalues are not sensitive to finite precision computations (for fixed $u$), or also~\cite[\S 3.3]{konstantinov2003perturbation}.
For discontinuous $f$, however, these subtleties alone cannot justify any accurately computed decimal digits.}
On the other hand, a small condition number is
 not a necessary condition for a small forward-backward error ratio: it is not inconceivable that certain ill-conditioned or even ill-posed problems can be solved accurately.
Consider, for example, the problem of computing an eigenvalue of the $4 \times 4$ matrix pencil (linear matrix polynomial)
\begin{equation}\label{eq:pencil-example}
L(x)=\begin{bmatrix}
-1&1&4&2\\
-2&3&12&6\\
1&3&11&6\\
2&2&7&4\end{bmatrix} x+\begin{bmatrix}2&-1&-5&-1\\
6&-2&-11&-2\\
5&0&-2&0\\
3&1&3&1\end{bmatrix};
\end{equation}
this is a singular matrix pencil (the determinant is identically zero) whose only finite eigenvalue is simple and equal to $1$ (see Section \ref{sec:eigencond} for the definition of an eigenvalue of a singular matrix polynomial and other relevant terminology). The input is $L(x)$ and the solution is $1$. If the QZ algorithm \cite{QZ}, which is the standard eigensolver for pencils, is called via {\tt MATLAB}'s command {\tt eig}~\footnote{MATLAB R2016a on Ubuntu 16.04}, the output is:
\begin{lstlisting}
>> eig(L0,-L1)

ans =

	  -138.1824366539536
	  -0.674131242894470
	   1.000000000000000
	   0.444114486065683
\end{lstlisting}
All but the third computed eigenvalues are complete rubbish. This is not surprising: singular pencils form a proper Zariski closed set in the space of matrix pencils of a fixed format, and it is unreasonable to expect that an unstructured algorithm would detect that the input is singular and return only one eigenvalue. Instead, being backward stable, QZ computes the eigenvalues of some nearby matrix pencil, and almost all nearby pencils have $4$ eigenvalues. On the other hand, the accuracy of the approximation of the genuine eigenvalue $1$ is quite remarkable. Indeed, the condition number of the problem that maps $L(x)$ to the exact eigenvalue $1$ is infinite because the map from matrix pencils to their eigenvalues is discontinuous at any matrix pencil whose determinant is identically zero. To make matters worse, there exist plenty of matrix pencils arbitrarily close to $L(x)$ and whose eigenvalues are \emph{all} nowhere near $1$. For example, for any $\epsilon > 0$, 
$$\hat{L}(x)=L(x)+\epsilon \left( \begin{bmatrix}0&-1&-4&-1\\
1&-3&-13&-3\\
0&-2&-8&-2\\
-1&-1&-3&-1\end{bmatrix} x + A\right),$$ 
where
$$A=\begin{bmatrix}-1&-1&-3&-2\\
-3&-3&-9&-6\\
-2&-2&-6&-4\\
-1&-1&-3&-2\end{bmatrix} \gamma_0 + \begin{bmatrix}1&0&0&0\\
3&0&0&0\\
2&0&0&0\\
1&0&0&0\end{bmatrix} \gamma_1 + \begin{bmatrix}0&-1&-4&-1\\
0&-3&-12&-3\\
0&-2&-8&-2\\
0&-1&-4&-1\end{bmatrix} \gamma_2 + \begin{bmatrix}0&0&0&0\\
-1&0&1&0\\
0&0&0&0\\
1&0&-1&0\end{bmatrix} \gamma_3,$$
has characteristic polynomial $\epsilon^2(\gamma_3-x)(x^3+\gamma_2 x^2 + \gamma_1 x + \gamma_0)$ and therefore, by an arbitrary choice of the parameters $\gamma_i$, can have eigenvalues literally \emph{anywhere}. Yet, unaware of this worrying caveat, the QZ algorithm computes an excellent approximation of the exact eigenvalue: $16$ correct digits! This example has not been carefully cherry-picked: readers are encouraged to experiment with any singular input in order to convince themselves that QZ often computes\footnote{Of course, if the exact solution is not known a priori, one faces the practical issue of deciding which of the computed eigenvalues are reliable. There are various ways in which this can be done in practice, such as artificially perturbing the problem; the focus of our work is on explaining why the correct solution has been shortlisted in the first place; see~\cite{hmp18} for a more practical perspective.} accurately the (simple) eigenvalues of singular pencils, or singular matrix polynomials, in spite of being a discontinuous problem. See also~\cite{hmp18} for more examples and a discussion of applications. Although the worst-case sensitivity to perturbations is indeed infinite, the raison d'\^{e}tre of the condition number, which is to predict the accuracy of computations on a computer, is not fulfilled.

Why does the QZ algorithm accurately compute the eigenvalue, when the map $f$ describing this computational problem is not even continuous? Two natural attempts at explaining this phenomenon would be to look at {\em structured condition numbers} and/or {\em average-case (stochastic) perturbation theory}. 

\begin{enumerate}
\item An algorithm is \emph{structured} if it computes the exact solution to a perturbed input, where the perturbations respect some special features 
of the input: for example singular, of rank $3$, triangular, or with precisely one eigenvalue. The vanilla implementation of QZ used here is unstructured in the sense that it does not preserve any of the structures that would explain the strange case of the algorithm that computes an apparently uncomputable eigenvalue.\footnote{There exist algorithms able to detect and exploit the fact that a matrix pencil is singular, such as the staircase algorithm~\cite{VanDooren1979}.} It does, however, preserve the real structure. In other words, if the input is real, QZ computes the eigenvalues of a nearby \emph{real} pencil. Yet, by taking real $\gamma_i$ in the example above, it is clear that there are real pencils arbitrary close to $L(x)$ whose eigenvalues are all arbitrarily far away from $1$.
\item The classical condition number is based on the worst-case perturbation of an input; as discussed in \cite[\S 2.8]{Higham1996}, this approach tends to be overly pessimistic in practice. Numerical analysis pioneer James Wilkinson, in order to illustrate that Gaussian elimination is unstable in theory but in practice its instability is only observed by mathematicians looking for it, is reported to have said~\cite{trefethen2012smart}
$$
\lowbiglquote Anyone \ that \ unlucky \ has \ already \ been \ run \ over \ by \ a \ bus.
\lowbigrquote 
$$
In other words: in Wilkinson's experience, the likelihood of seeing the admittedly terrifying worst case appeared to be very small, and therefore Wilkinson believed that being afraid of the potential catastrophic instability of Gaussian elimination is an irrational attitude.
Based on this experience, Weiss et al.~\cite{weiss1986average}
and Stewart~\cite{stewart1990stochastic}
 proposed to study the effect of perturbations {\em on average}, as opposed to worst-case; see~\cite[\S 2.8]{Higham1996} for more references on work addressing the stochastic analysis of roundoff errors. This idea was later formalized and developed further by by Armentano~\cite{armentano2010stochastic}. This approach gives some hope to explain the example above, because it is known that the set of perturbations responsible for the discontinuity of $f$ has measure zero \cite{dd09}. However, this does not imply that on average perturbations are not harmful. In fact, as we will see, the stochastic condition number for the example above (or for similar problems) is still infinite! Average-case perturbation analysis, at least in the form in which it has been used so far, is still unable to solve the puzzle.
\end{enumerate}

While neither structured nor average-case perturbation theory can explain the phenomenon observed above, Wilkinson's colourful quote does contain a hint on how to proceed: shift  attention from average-case analysis of perturbations to bounding rare events. We will get back to the matrix pencil~\eqref{eq:pencil-example} in Example~\ref{ex:L-pencil-ex1}, where we show that our new theory does explain why this problem is solved to high accuracy using standard backward stable algorithms. 

In summary, the main contributions of this paper are

\begin{enumerate}
\item a new species of `weak' condition numbers, which we call the \emph{weak worst-case condition number} and the \emph{weak stochastic condition number}, that give a more accurate description of the perturbation behaviour of a computational map (Section~\ref{sec:strong});
\item a precise probabilistic analysis of the sensitivity of the problem of computing simple eigenvalues of singular matrix polynomials (Sections~\ref{sec:probanal} and~\ref{sec:main});
\item an illustration of the advantages of the new concept by demonstrating that, unlike both classical and stochastic condition numbers, the weak condition numbers \emph{are} able to explain why the apparently uncomputable eigenvalues of singular matrix polynomials, such as the eigenvalue $1$ in the example above, can be computed with remarkable accuracy (Example~\ref{ex:L-pencil-ex1});
\item a concrete method for bounding the weak condition numbers for the eigenvalues of singular matrix polynomials (Section~\ref{sec:compweak}).
\end{enumerate}

\subsection{Related work}
Rounding errors, and hence the perturbations considered, are not random~\cite[1.17]{Higham1996}. Nevertheless, the observation that the computed bounds on rounding errors are overly pessimistic has led to the study of statistical and probabilistic models for rounding errors.
An early example of such a statistical analysis is Goldstine and von Neumann~\cite{goldstine1951numerical}, see ~\cite[2.8]{Higham1996} and the references therein for more background. 
Recently, Higham and Mary~\cite{higham2018new} obtained probabilistic rounding error bounds for a wide variety of algorithms in linear algebra. In particular, they give a rigorous foundation to Wilkinson's rule of thumb, which states that constants in rounding error bounds can be safely replaced by their square roots. 

The idea of using an average, rather than a supremum, in the definition of conditioning was introduced by Weiss et.al.~\cite{weiss1986average} in the context of the (matrix) condition number of solving systems of linear equations, and a more comprehensive stochastic perturbation theory was developed by G.~W.~Stewart~\cite{stewart1990stochastic}. In~\cite{armentano2010stochastic}, Armentano introduced the concept of a smooth condition number, and showed that it can be related to the worst case condition. His work uses a geometric theory of conditioning and does not extend to singular problems. 

The line of work on random perturbations is not to be confused with the {\em probabilistic analysis of condition numbers}, where a condition number is a given function, and the distribution of this function is studied over the space of inputs (see~\cite{BC2013} and the references therein). Nevertheless, our work is inspired by the idea of weak average-case analysis~\cite{amelunxen2017average} that was developed in this framework. Weak average-case analysis is based on the observation, which has origins in the work of Smale~\cite{smale1981fundamental} and Kostlan~\cite{kostlan1988complexity}, that discarding a small set from the input space can dramatically improve the expected value of a condition number, shifting the focus away from the {\em average case} and
towards bounding the probability of rare events.
Our contribution is to apply this line of thought to study random perturbations instead of random inputs. However, we stress that we do not seek to model the distribution of perturbations. The aim is to formally quantify statements such as ``the set of bad perturbations is small compared to the set of good perturbations''. In other words, the (non-random) accumulation of rounding errors in a procedure would need a very good reason to give rise to a badly perturbed problem. 

The conditioning of regular polynomial eigenvalue problems has been studied in detail by Tisseur~\cite{tisseur2000backward} and by Dedieu and Tisseur in a homogeneous setting~\cite{dedieu2003perturbation}. 
A probabilistic analysis of condition numbers (for random inputs) for such problems was given by Armentano and Beltr\'an~\cite{armentano2017polynomial} over the complex numbers and by Beltr\'an and Kozhasov~\cite{beltran2018real} over the real numbers. Their work studies the distribution of the condition number on the whole space of inputs, and such an analysis only considers the condition number of regular matrix polynomials.
A perturbation theory for singular polynomial eigenvalue problems was developed by de Ter\'an and Dopico~\cite{dd10}, and our work makes extensive use of their results. A method to solve singular generalized eigenvalue problems with plain QZ, based on applying a certain perturbation to them, is proposed in \cite{hmp18} (see also the references therein); note that our work goes beyond this, by showing how to estimate the weak condition number that could guarantee, often with overwhelming probability, that QZ will do fine even without any preliminary perturbation step.

\subsection{Organization of the paper}
The paper is organized as follows. In Section \ref{sec:strong} we review the rigorous definitions of the worst-case (von Neumann-Turing) condition number and the stochastic framework (Weiss et.al., Stewart, Armentano), and comment on their advantages and limitations. We then define the weak condition numbers as quantiles and argue that, even when Wilkinson's metaphorical bus hits von Neumann-Turing's and Armentano-Stewart's theories of conditioning, ours comes well endowed with powerful dodging skills. 
In Section~\ref{sec:eigencond} we introduce the perturbation theory of singular matrix polynomials, along with the definitions of simple eigenvalues and eigenvectors. We define the input-output map underlying our case study and introduce 
the directional sensitivity of such problems.
In Section \ref{sec:probanal}, which forms the core of this paper, we carry out a detailed analysis of the probability distribution of the directional sensitivity of the problems introduced in Section~\ref{sec:eigencond}.
In Section~\ref{sec:main}, we translate the probabilistic results from Section~\ref{sec:probanal} into the language of weak condition numbers and prove the main results, Theorem~\ref{thm:main-complex} and Theorem~\ref{thm:main-real}. In Section~\ref{sec:compweak} we sketch how our new condition numbers can be estimated in practice. Along the way we derive a simple concentration bound on the directional sensitivity of regular polynomial eigenvalue problems.
Finally, in Section \ref{sec:conclusions}, we give some concluding remarks and discuss potential further applications.

\section{Theories of conditioning}\label{sec:strong}
For our purposes, a computational problem is a map between normed vector spaces~\footnote{One can, more generally, allow $\mathcal{V}$ and $\mathcal{W}$ to be anything with a notion of distance, such as general metric spaces or Riemannian manifolds. All the definitions of condition can be adapted accordingly; in this paper we focus on the case of normed vector spaces. We will also only need such a map to be defined locally near an input of interest.}
$$f:  \mathcal{V} \rightarrow \mathcal{W}, \qquad  D \mapsto S :=f(D),$$
and we will denote the (possibly different) norms in each of these spaces by $\|\cdot\|$.
 Following the remark on \cite[p. 56]{High:FM}, for simplicity of exposition in this paper we focus on absolute, as opposed to relative, condition numbers. The condition numbers considered depend on the map $f$ and an input $D\in \mathcal{V}$.

As we are only concerned with the condition of a fixed computational problem at a fixed input $D$, in what follows we omit reference to $f$ and $D$ in the notation.

\begin{definition}[Worst-case condition number]\label{def:turing}
The condition number of $f$ at $D$ is
$$\kappa = \lim_{\epsilon \rightarrow 0} \sup_{ \|E\|\leq 1} \frac{\|  f(D+\epsilon E) - f(D) \|}{\epsilon\|E\|} .$$
\end{definition}

If $f$ is Fr\'{e}chet differentiable at $D$, then this definition is equivalent to the operator norm of the Fr\'{e}chet derivative of $f$. 
However, Definition~\ref{def:turing} also applies (and can even be finite) when $f$ is not differentiable. In complexity theory \cite{BCSS1998,BC2013}, an elegant geometric definition of condition number is often used, which is essentially equivalent to Defintion \ref{def:turing} under certain assumptions (which include smoothness).

The following definition is loosely derived from the work of Stewart~\cite{stewart1990stochastic} and Armentano~\cite{armentano2010stochastic}, based on earlier work by Weiss et. al.~\cite{weiss1986average}. In what follows, we use the terminology $X\sim \mathcal{D}$ for a random variable with distribution $\mathcal{D}$, and $\Expect_{X\sim \mathcal{D}}[\cdot]$ for the expectation with respect to this distribution.

\begin{definition}[Stochastic condition number]\label{def:armentano}
Let $E$ be a $\mathcal{V}$-valued random variable with distribution $\mathcal{D}$ and assume that $\Expect_{E\sim \mathcal{D}}[E]=0$ and $\Expect_{E\sim \mathcal{D}}[\|E\|^2]=1$. Assume that the function $f$ is measurable.
Then the stochastic condition number is
$$\kappa_s = \lim_{\epsilon \rightarrow 0} \mathbb{E}_{E \sim \mathcal{D}} \frac{\|f(D+\epsilon E)-f(D)\|}{\epsilon \|E\|}.$$
\end{definition}

\begin{remark}
We note in passing that Definition \ref{def:armentano} depends on the choice of a measure $\mathcal{D}$. This measure is a parameter that the interested mathematician should choose as convenient; this is of course not particularly different than the freedom one is given in picking a norm. In fact,
it is often convenient to combine these two choices, using a distribution that is invariant with respect to a given norm. Typical choices that emphasize invariance are the uniform (on a sphere) or Gaussian distributions, and the Bombieri-Weyl inner product when dealing with homogeneous multivariate  polynomials~\cite[16.1]{BC2013}. Technically speaking, the distribution is on the {\em space of perturbations}, rather than the space of inputs.
\end{remark}

If $f$ is differentiable at $D$ and $\Input$ is finite dimensional,
 then it was observed by Armentano~\cite{armentano2010stochastic} that the stochastic condition number can be related to the worst-case one. We illustrate this relation in a simple but instructive special case. Consider the  setting\footnote{Armentano's results apply to differentiable maps between Riemannian manifolds, and cover the moments of the directional derivative as well: they are stronger and are derived with a more comprehensive approach.} 
 where $f\colon \R^m\to \R^n$ ($m\geq n$)
is differentiable at $D\in \R^m$, so that $\kappa$ is the operator norm of the differential.
If $\sigma_1\geq \cdots \geq \sigma_m$ denote the singular values of $\diff{f}(D)$ (with $\sigma_i=0$ for $i>n$), then $\kappa = \sigma_1$. 
If $\mathcal{D}$ is the uniform distribution on the sphere, then
\begin{equation}\label{eq:armentano-bound}
  \frac{1}{m}\kappa \stackrel{(a)}{\leq} \sigma_1 \Expect_{E\sim D}|E_1| \leq \Expect_{E\sim D}\left[\sqrt{\sum_{i}\sigma_i^2E_i^2}\right] \stackrel{(b)}{=} \Expect_{E\sim D} [\|\diff{f}(D)E\|_2] = \kappa_s,
\end{equation}
where for (a) we used the fact that 
\begin{equation*}
\Expect_{E\sim \mathcal{D}}|E_1| = \frac{1}{m}\Expect_{E\sim \mathcal{D}}\|E\|_1\geq \frac{1}{m}\Expect_{E\sim \mathcal{D}} \|E\|_2 = \frac{1}{m}
\end{equation*}
and for (b) we used the orthogonal invariance of the uniform distribution on the sphere.
As we will see in the case of singular polynomial eigenvalue problems with complex perturbations, the bound~\eqref{eq:armentano-bound} does not hold in general, as 
the condition number can be infinite while the stochastic condition number is bounded.
However, sometimes it can happen that the stochastic condition number is also infinite, because the ``directional sensitivity'' (see Definition~\ref{def:ds}) is not an integrable function. For example, for the problem of computing the eigenvalue of the singular pencil $L(x)$ in the introduction, in spite of the fact that real perturbations are analytic for all but a proper Zariski closed set of perturbations~\cite{dd09}, when restricting to real perturbations, we get
$$\kappa_s = \kappa = \infty.$$
Despite this, QZ computes the eigenvalue $1$ with $16$ digits of accuracy.

To remedy the shortcomings of the stochastic condition number as defined in~\ref{def:armentano}, we propose a change in focus from the expected value to tail bounds and quantiles, and the key concept for that purpose is the directional sensitivity.
Just as the classical worst-case condition corresponds to the norm of the derivative, the directional sensitivity corresponds to a directional derivative. And, just as a function can have some, or all, directional derivatives while still not being continuous, a computational problem can have well-defined directional sensitivities but have infinite condition number.

\begin{definition}[Directional sensitivity]\label{def:ds}
The \emph{directional sensitivity} of the computational problem $f$ at the input $D$ with respect to the perturbation $E$ is
$$\sigma_E=\lim_{\epsilon \rightarrow 0} \frac{\|f(D+\epsilon E)- f(D)\|}{\epsilon \|E\|}.$$
\end{definition}

The directional sensitivity takes values in $[0,\infty]$.
In numerical analytic language, the directional sensitivity is the limit, for a particular direction of the backward error, of the ratio of forward and backward errors of the computational problem $f$; this limit is taken letting the backward error tend to zero (again having fixed its direction), which could also be thought of as letting the unit round-off tend to zero. See, e.g., \cite[\S 1.5]{Higham1996} for more details on this terminology.

The directional sensitivity is, if it is finite, $\|E\|^{-1}$ times the norm of the G\^ateaux derivative $\diff{f}(D;E)$ of $f$ at $D$ in direction $E$. If $f$ is Fr\'echet differentiable, then the G\^ateaux derivative agrees with the Fr\'echet derivative, and we get
\begin{equation}\label{eq:worst-case-direction}
  \kappa = \sup_{\|E\| \leq 1} \sigma_E.
\end{equation}

If $E$ is a $\mathcal{V}$-valued random variable satisfying the conditions of Definition~\ref{def:armentano} and
if $f$ is G\^ateaux differentiable in almost all directions, then by the Fatou-Lebesgue Theorem we get
\begin{equation*}
  \kappa_s = \Expect[\sigma_E].
\end{equation*}

When integrating, null sets can be safely ignored; however, depending on the exact nature of the divergence (or lack thereof) of the integrand when approaching those null sets, the value of the integral need not be finite. To overcome this problem and still give probabilistically meaningful statements, we propose to use instead the concept of \emph{numerical null sets}, i.e., sets of finite but small (in a sense that can be made precise depending on, for example, the unit roundoff of the number system of choice, the confidence level required by the user, etc.) measure. This is analogous to the idea that the ``numerical zero'' is the unit roundoff. 
We next define our main characters, two classes of weak condition numbers which generalize, respectively, the classical worst-case and stochastic condition numbers. 

In the following, we fix a probability space $(\Omega,\Sigma,\mathbb{P})$ and a random variable $E\colon \Omega\to \mathcal{V}$, where we consider $\mathcal{V}$ endowed with the Borel $\sigma$-algebra. We further assume that
\begin{equation*}
 \Expect[E] = \int_{\Omega} E(\omega) \ \diff{\Prob}(\omega) = 0, \quad \Expect[\|E\|^2] = \int_{\Omega} \|E(\omega)\|^2 \ \diff{\Prob}(\omega) = 1.
\end{equation*}
The following definitions assume that $\sigma_E$ is $\Prob$-measurable. This is the case, for example, if $f$ is measurable and the directional (G\^ateaux) derivative $\mathrm{d}f(D;E(\omega))$ exists $\Prob$-a.e.

\begin{definition}[Weak worst-case and weak stochastic condition number]\label{def:wwc}
Let $0\leq \delta <1$ and assume that $\sigma_E$ is $\Prob$-measurable. 
The $\delta$-weak worst-case condition number and the $\delta$-weak stochastic condition number are defined as
\begin{equation*}
\kappa_w(\delta) := \inf \{y \in \R \colon \Prob\{ \sigma_E < y \} \geq 1-\delta\}, \quad \quad \kappa_{ws}(\delta) := \Expect[\sigma_E \ | \ \sigma_E \leq \kappa_w(\delta)].
\end{equation*}
\end{definition}

\begin{remark} We note that one can give a definition of the weak worst-case and weak stochastic condition number that does not require $\sigma_E$ to be a random variable, by setting
\begin{equation*}
\kappa_w(\delta) = \inf_{\substack{\SS \in \Sigma,\\ |\SS| \geq 1-\delta}} \sup_{\omega \in \SS} \sigma_{E(\omega)}, \quad \quad \quad\kappa_{ws}(\delta) = \inf_{\substack{\SS \in \Sigma,\\ |\SS| \geq 1-\delta}} \Expect[\sigma_E \ | \ \mathcal{S}],
\end{equation*}
where we used the notation $|\mathcal{S}|=\mathbb{P}(\mathcal{S})$ for the measure of a set if there is no ambiguity. This form is reminiscent of the definition of weak average-case analysis in~\cite{amelunxen2017average}, and when $\sigma_E$ is a random variable it can be shown to be equivalent to \ref{def:wwc}. Moreover, this slightly more general definition better illustrates the essence of the weak condition numbers: these are the (worst-case and average-case) condition numbers that ensue when one is allowed to discard a ``numerically invisible'' subset from the set of perturbations.
\end{remark}

The directional sensitivity has an interpretation as (the limit of) a ratio of forward and backward errors, and hence the new approach provides a potentially useful general framework to give probabilistic bounds on the forward accuracy of outputs of numerically stable algorithms. Moreover, as we will discuss in Section~\ref{sec:compweak}, upper bounds on the weak condition numbers can be computed in practice for a natural distribution. One can therefore see $\delta$ as a parameter representing the confidence level that a user wants for the output, and any computable upper bound on $\kappa_w$ becomes a practical reliability measure on the output, valid with probability $1-\delta$. Although of course round-off errors are not really random variables, we hope that modelling them as such can become, with this ``weak theory'', a useful tool for numerical analysis problems whose traditional condition number is infinite. 

\section{Eigenvalues of matrix polynomials and their directional sensitivity}\label{sec:eigencond}
Algebraically, the spectral theory of matrix polynomials is most naturally described over an algebraically closed field; however, the theory of condition is analytic in nature and it is sometimes of interest to restrict the coefficients, and their perturbations, to be real. In this section we give a unified treatment of both real and complex matrix polynomials. For conciseness we keep this overview very brief; interested readers can find further details in~\cite{dd10,dd09,Dopico2018,GLR09,Noferini2012} and the references therein.
A matrix polynomial is a matrix $P(x)\in \F[x]^{n\times n}$, where $\F\in \{\C,\R\}$ is a field. Alternatively, we can think of it as an expression
\begin{equation*}
  P(x) = P_0+P_1x+\cdots +P_dx^d,
\end{equation*}
with $P_i\in \F^{n\times n}$. If we require $P_d \neq 0$, 
then the integer $d$ in such an expression is called the {\em degree} of the matrix polynomial \footnote{By convention, the zero matrix polynomial has degree $-\infty$.}. We denote the vector space of matrix polynomials over $\F$ of degree at most $d$ by $\F^{n\times n}_d[x]$. A matrix polynomial is called {\em singular} if $\det P(x)\equiv 0$ and otherwise regular.
An element $\lambda \in \C$ is said to be a finite eigenvalue of $P(x)$ if
\begin{equation*}
  \rank_{\C}(P(\lambda)) < \rank_{\F(x)}(P(x)) =: r,
\end{equation*}
where $\F(x)$ is the field of fractions of $\F[x]$, that is, the field of rational functions with coefficients in $\F$. We assume throughout rank $r\geq 1$ (which implies $n\geq 1$) and degree $d\geq 1$. 
The \emph{geometric multiplicity} of the eigenvalue $\lambda$ is the amount by which the rank decreases in the above definition,
$$g_{\lambda} =  r- \rank_{\C}(P(\lambda)).$$

There exist matrices $U, V \in \F[x]^{n\times n}$ with $\det(U)\in \F\backslash\{0\}$, $\det(V)\in \F\backslash\{0\}$, that transform $P(x)$ into its {\em Smith canonical form},
\begin{equation}\label{eq:smith}
 U^*P(x)V=D:=\diag(h_1(x),\dots,h_r(x),0,\dots,0),
\end{equation}
where the \emph{invariant factors} $h_i(x)\in \F[x]$ are non-zero monic polynomials such that $h_i(x)|h_{i+1}(x)$ for $i\in \{1,\dots,r-1\}$. If one has the factorizations $h_{i}=(x-\lambda)^{k_i}\tilde{h}_{i}(x)$ for some $\tilde{h}_i(x)\in \C[x]$, with $0\leq k_i\leq k_{i+1}$ for $i\in \{1,\dots,r-1\}$ and $(x-\lambda)$ not dividing any of the $\tilde{h}_i(x)$, then the $k_i$
are called the \emph{partial multiplicities} of the eigenvalue $\lambda$. The \emph{algebraic multiplicity} $a_\lambda$ is the sum of the partial multiplicities. Note that an immediate consequence of this definition is $a_\lambda \geq g_\lambda$. If $a_\lambda=g_\lambda$ (i.e., all non-zero $k_i$ equal to $1$) then the eigenvalue $\lambda$ is said to be \emph{semisimple}, otherwise it is \emph{defective}. If $a_\lambda=1$ (i.e., $k_i=1$ for $i=r$ and zero otherwise), then we say that $\lambda$ is \emph{simple}, otherwise it is \emph{multiple}.
 
A square matrix polynomial is \emph{regular} if $r=n$, i.e., if $\det P(x)$ is not identically zero. A finite eigenvalue of a regular matrix polynomial is simply a root of the characteristic equation $\det P(x)=0$, and its algebraic multiplicity is equal to the multiplicity of the corresponding root. If a matrix polynomial is not regular it is said to be \emph{singular}. More generally, a finite eigenvalue of a matrix polynomial (resp. its algebraic multiplicity) is a root (resp. the multiplicity as a root) of the equation $\gamma_r(x)=0$, where $\gamma_r(x)$ is the monic greatest common divisor of all the minors of $P(x)$ of order $r$ (note that $\gamma_n(x)=\det P(x)$). 

\begin{remark}
The concept of an eigenvalue, and the other definitions recalled here, are valid also in the more general setting of rectangular matrix polynomials. However, in that scenario a generic matrix polynomial has no eigenvalues \cite{dd09}; as a consequence, a perturbation of a matrix polynomial with an eigenvalue would almost surely remove it. This is a fairly different setting than in the square case, and a deeper probabilistic analysis of the rectangular case beyond the scope of the present paper.

We mention in passing that there are possible ways to extend the analysis to the rectangular case, such as embedding them in a larger square matrix polynomial or (at least in the case of pencils, or linear matrix polynomials) consider structured perturbations that do preserve eigenvalues. 
\end{remark}

\subsection{Eigenvectors}\label{sub:ev} To define the eigenvectors, let $\{b_1(x),\dots,b_{n-r}(x)\}$ and $\{c_1(x),\dots,c_{n-r}(x)\}$ be minimal bases~\cite{Dopico2018,Forney1975,Noferini2012} of $\ker P(x)$ and $\ker P(x)^*$ (as vector spaces over $\F(x)$), respectively. For $\lambda\in \C$, it is not hard to see \cite{Dopico2018,Noferini2012} that
 $\ker_\lambda P(x):=\mathrm{span}\{b_1(\lambda),\dots,b_{n-r}(\lambda)\}$ and $\ker_\lambda P(x)^*:=\mathrm{span}\{c_1(\lambda^*),\dots,c_{n-r}(\lambda^*)\}$ 
 are vector spaces over $\C$ of dimension $n-r$. 
 
Note that $\ker_\lambda P(x) \subseteq \ker P(\lambda)$ and
$\ker_\lambda P(x)^*\subseteq \ker P(\lambda)^*$ for $\lambda\in \C$, and that the difference in dimension is the geometric multiplicity, $\ker P(\lambda)-\ker_\lambda P(x) = \ker P(\lambda)^*- \ker_\lambda P(x)^* = g_\lambda$.
A {\em right eigenvector} corresponding to an eigenvalue $\lambda\in \C$ is defined \cite[Sec. 2.3]{Dopico2018} to be a nonzero element of the quotient space $\ker P(\lambda)/\ker_\lambda P(x)$. 
A {\em left eigenvector} is similarly defined as an element of $\ker P(\lambda)^*/\ker_\lambda P(x)^*$. In terms of the Smith canonical form~\eqref{eq:smith},
the last $n-r$ columns of $U$, evaluated at $\lambda^*$, represent a basis of $\ker_\lambda P(x)^*$, while the last $(n-r)$ columns of $V$, evaluated at $\lambda$, represent a basis of $\ker_\lambda P(x)$. 

In the analysis we will be concerned with a quantity of the form $|u^*P'(\lambda) v|$, where $u,v$ are representatives of eigenvectors. It is known \cite[Lemma 2.9]{Dopico2018} that $b \in \ker_\lambda P(x)$ is equivalent to the existence of a polynomial vector $b(x)$ such that $b(\lambda)=b$ and $P(x)b(x)=0$. Then,
\begin{equation*}
  0 = \frac{\mathrm{d}}{\diff{x}}P(x)b(x)|_{x=\lambda} = P'(\lambda) b(\lambda)+ P(\lambda)b'(\lambda)
\end{equation*}
implies that for any representative of a left eigenvector $u\in \ker P(\lambda)^*$ we get $u^*P'(\lambda)b(\lambda)=0$. 
It follows that for an eigenvalue representative $v$, $u^*P'(\lambda)v$ depends only the component of $v$ orthogonal to $\ker_\lambda P(x)$, and an analogous argument also shows that this expression only depends on the component of $u$ orthogonal to $\ker_\lambda P(x)^*$. 
In practice, we will therefore choose representatives $u$ and $v$ for the left and right eigenvalues that are orthogonal to $\ker_{\lambda} P(x)^*$ and $\ker_\lambda P(x)$, respectively, and have unit norm. If $P(x)\in \F^{n\times n}_d[x]$ is a matrix polynomial with simple eigenvalue $\lambda$, then there is a unique (up to sign) way of choosing such representatives $u$ and $v$. 

\subsection{Perturbations of singular matrix polynomials: the De Ter\'{a}n-Dopico formula}\label{sec:deteran-dopico}
Assume that $P(x)\in \F^{n\times n}_d[x]$, where $\F \in \{ \R, \C\}$, is a matrix polynomial of rank $r\leq n$, and let $\lambda$ be a simple eigenvalue.
Let $X=[U\ u]\in \C^{n\times (n-r+1)}$ be a matrix whose columns form a basis of $\ker P(\lambda)^*$, and such that the columns of $U\in \C^{n\times (n-r)}$ form a basis of $\ker_\lambda P(x)^*$.
Likewise, let $Y=[V \ v]$ 
be a matrix whose columns form a basis of $\ker P(\lambda)$, and such that the columns of $V\in \C^{n\times (n-r)}$ form a basis of $\ker_\lambda P(x)$. In particular, $v$ and $u$ are representatives of, respectively, right and left eigenvectors of $P(x)$. 
The following explicit characterization of a simple eigenvalue is due to De Ter\'an and Dopico~\cite[Theorem 2 and Eqn. (20)]{dd10}. To avoid making a case distinction for the regular case $r=n$, we agree that $\det(U^*E(\lambda)V)=1$ if $U$ and $V$ are empty.

\begin{theorem}\label{thm:dd-2}
Let $P(x)\in \F^{n\times n}_d[x]$ be matrix polynomial of rank $r$ with simple eigenvalue $\lambda$ and $X,Y$ as above. Let $E(x)\in \F^{n\times n}_d[x]$ be such that $X^*E(\lambda)Y$ is non-singular. Then for small enough $\epsilon>0$, the perturbed matrix polynomial $P(x)+\epsilon E(x)$ has exactly one eigenvalue $\lambda(\epsilon)$ of the form
\begin{equation}\label{eq:perturbed-eigen}
\lambda(\epsilon) = \lambda-\frac{\det(X^*E(\lambda)Y)}{u^*P'(\lambda)v \cdot \det(U^*E(\lambda)V)}\epsilon + O(\epsilon^2).
\end{equation}
\end{theorem}

Note that in the special case $r=n$ we recover the expression for regular matrix polynomials from~\cite[Theorem 5]{tisseur2000backward} and~\cite[Corollary 1]{dd10},
\begin{equation}\label{eq:fran}
 \lambda(\epsilon) = \lambda-\frac{u^*E(\lambda)v}{u^*P'(\lambda)v}\epsilon + O(\epsilon^2),
\end{equation}
where $u,v$ are left and right eigenvectors corresponding to the eigenvalue $\lambda$. 

\subsection{The directional sensitivity of a singular polynomial eigenproblem}\label{sec:def-input-output}
We can now describe the input-output map that underlies our analysis. By the local nature of our problem, we consider a fixed matrix polynomial $P(x)\in \F^{n\times n}_d[x]$ of rank $r$ with simple eigenvalue $\lambda$, 
and define the input-output function 
\begin{equation*}
  f\colon \F^{n\times n}_d[x]\to \C
\end{equation*}
that maps $P(x)$ to $\lambda$, maps $P(x)+\epsilon E(x)$ to $\lambda(\epsilon)$ for any $E(x)$ and $\epsilon>0$ satisfying the conditions of Theorem~\ref{thm:dd-2}, and maps any other matrix polynomial to an arbitrary number other than $\lambda$.

An immediate consequence of Theorem~\ref{thm:dd-2} and our definition of the input-output map is an explicit expression for the directional sensitivity of the problem. Here we write $\|E\|$ for the Euclidean norm of the vector of coefficients of $E(x)$ as a vector in $\F^{n^2(d+1)}$. From now on, when talking about the ``directional sensitivity of an eigenvalue in direction $E$'', we implicitly refer to the input-output map $f$ defined above.

\begin{corollary}\label{cor:directional}
Let $\lambda$ be a simple eigenvalue of $P(x)$ and let $E(x)\in \F^{n\times n}_d[x]$ be a regular matrix polynomial. Then the directional sensitivity of the eigenvalue $\lambda$ in direction $E(x)$ is 
\begin{equation}\label{eq:directional-sensitivity}
  \sigma_E = \frac{1}{\|E\|}\left|\frac{\det(X^*E(\lambda)Y)}{u^*P'(\lambda)v \cdot \det(U^*E(\lambda)V)}\right|.
\end{equation}
In the special case $r=n$, we have
 \begin{equation}\label{eq:directional-sensitivity-regular}
  \sigma_E = \frac{1}{\|E\|}\left|\frac{u^*E(\lambda) v}{u^*P'(\lambda)v}\right|.
\end{equation}
\end{corollary}

For the goals in this paper, these results suffice. However, we note that it is possible to obtain equivalent formulae for the expansion that, unlike the one by De Ter\'{a}n and Dopico, do not involve the eigenvectors of singular polynomials.

Finally, we introduce a parameter that will enter all of our results, and coincides with the inverse of the worst-case condition number in the regular case $r=n$. Choose representatives $u,v$ of the eigenvectors that satisfy $\|u\|=\|v\|=1$ and (if $r<n$) $U^*u=V^*v=0$.
For such a choice of eigenvectors, define
\begin{equation}\label{eq:cdef}
  \gamma_P := |u^*P'(\lambda) v| \cdot \left(\sum_{j=0}^d |\lambda|^{2j}\right)^{-1/2}.
\end{equation}

We conclude with the following variation of~\cite[Theorem 5]{tisseur2000backward}. For a proof of the following result, see~\cite[Lemma 2.1]{adhikari2011structured} or~\cite{alam2019sensitivity} for a discussion in a wider context.

\begin{proposition}\label{prop:cond-regular}
Let $P(x)\in \F^{n\times n}_d[x]$ be a regular matrix polynomial and $\lambda\in \C$ a simple eigenvalue. Then the worst-case condition number of the problem of computing $\lambda$ is $\kappa = \gamma_P^{-1}$.
\end{proposition}

\begin{remark}
In practice, an algorithm such as QZ applied to $P(x)$ will typically compute {\em all} the eigenvalues of a nearby matrix polynomial. Therefore, any conditioning results on the conditioning of our specific input-output map $f$ will explain why the correct eigenvalue is found {\em among} the computed eigenvalues, but not tell us how to choose the right one in practice. For selecting the right eigenvalue one could use heuristics, such as computing the eigenvalues of an artificially perturbed problem.
For more details on these practical considerations, we refer to~\cite{hmp18}.
\end{remark}

\section{Probabilistic analysis of the directional sensitivity}\label{sec:probanal}
In this section we study the probability distribution of the directional sensitivity of a singular polynomial eigenvalue problem
To deal with real and complex perturbations simultaneously as far as possible, we follow the convention from random matrix theory~\cite{edelman2005random} and parametrize our results with a parameter $\beta$, where $\beta=1$ if $\F=\R$ and $\beta=2$ if $\F=\C$. 
We consider perturbations $E(x)=E_0+E_1x+\cdots+E_dx^d$, which we identify with the matrix $E=\begin{bmatrix} E_0 & \cdots & E_d\end{bmatrix}\in \mathbb{F}^{n\times n(d+1)}$  (each $E_i\in \F^{n\times n}$), and denote by $\|E\|$ the Euclidean norm of $E$ considered as a vector in $\F^N$, where $N:=n^2(d+1)$ (equivalently, the Frobenius norm of the matrix $E$).
When we say that $E$ is uniformly distributed on the sphere, written $E\sim \mathcal{U}(\beta N)$ with $\beta=1$ for real perturbations and and $\beta=2$ if $E$ is complex, we mean that the image of $E$ under an identification $\mathbb{F}^{n\times n(d+1)}\cong \R^{\beta N}$ is uniformly distributed on the corresponding unit sphere $S^{\beta N-1}$.
To avoid trivial special cases, we assume that $r\geq 1$ and $d\geq 1$, so that, in particular, $N\geq 2$. 

The following theorem characterizes the distribution of the directional sensitivity under uniform perturbations.

\begin{theorem}\label{thm:main-distr}
Let $P(x)\in \F_d^{n\times n}[x]$ be a matrix polynomial of rank $r$ and let $\lambda$ be a simple eigenvalue of $P(x)$. If $E\sim \uniform(\beta N)$, where $\beta=1$ if $\F=\R$ and $\beta=2$ if $\F=\C$, then the directional sensitivity of $\lambda$ in direction $E(x)$ satisfies
\begin{equation*}
 \Prob\{\sigma_E\geq t\} = \begin{cases} \Prob\{Z_N/Z_{n-r+1} \geq \gamma_P^2 t^2\} & \text{ if } r<n\\
  \Prob\{Z_N \geq \gamma_P^2 t^2\} & \text{ if } r=n
  \end{cases},
\end{equation*}
where $Z_k\sim \Beta(\beta/2,\beta(k-1)/2)$ denotes a beta distributed random variable with parameters $\beta/2$ and $\beta(k-1)/2$, and $Z_N$ and $Z_{n-r+1}$ are independent.
\end{theorem}

The proof is given later in this section, after having introduced some preliminary concepts and results.
If $r=n$, then the directional sensitivity is distributed like the square root of a beta random variable, and in particular it is bounded. 
Using the density of the beta distribution, we can derive the moments and tail bounds for the distribution of the directional sensitivity explicitly. 

\begin{corollary}\label{cor:expected}
Let $P(x)\in \F_d^{n\times n}[x]$ be a matrix polynomial of rank $r$ and let $\lambda$ be a simple eigenvalue of $P(x)$. If $E\sim \uniform(\beta N)$, where $\beta=1$ if $\F=\R$ and $\beta=2$ if $\F=\C$, then the expected directional sensitivity of $\lambda$ in direction $E(x)$ is
\begin{equation*}
  \Expect[\sigma_E] = \begin{cases}
  \frac{1}{\gamma_P}\frac{\pi}{2}\frac{\Gamma(N)\Gamma(n-r+1)}{\Gamma(N+1/2)\Gamma(n-r+1/2)} & \text{ if } \F=\C\\
  \infty & \text{ if } \F=\R \text{ and } r<n\\
  \frac{1}{\gamma_P}\frac{\Gamma(N/2)}{\sqrt{\pi}\Gamma((N+1)/2)} &  \text{ if } \F=\R \text{ and } r=n.
  \end{cases}
\end{equation*}
If $t\geq \gamma_P^{-1}$, then for $r<n$ we have the tail bounds
\begin{equation}\label{eq:tailbounds}
  \Prob\{\sigma_E\geq t\} \leq \begin{cases}
  \frac{1}{\gamma_P^2}\frac{n-r}{N} \frac{1}{t^2} & \text{ if } \F=\C\\
  \frac{1}{\gamma_P}\frac{2}{\pi}\frac{\Gamma(N/2)\Gamma((n-r+1)/2)}{\Gamma((N+1)/2)\Gamma((n-r)/2)} \frac{1}{ t} & \text{ if } \F=\R.
  \end{cases}
\end{equation}
If $r=n$, then $\sigma_E\leq \gamma_P$. 
\end{corollary}

\begin{proof}
For the expectation, using Theorem~\ref{thm:main-distr} in the case $r<n$, we have
\begin{equation*}
 \Expect[\sigma_E] = \frac{1}{\gamma_P} \int_{0}^\infty \Prob\{(Z_N/Z_{n-r+1})^{1/2}\geq t\} \ \diff{t} = \frac{1}{\gamma_P} \Expect[(X_N/X_{n-r+1})^{1/2}],
\end{equation*}
where $X_k$ denotes a $\Beta(\beta /2,\beta(k-1)/2)$ distributed random variable.
The claimed tail bounds and expected values for $r<n$ follow by applying Lemma~\ref{le:betaexp} with $k=2$, $a=c=\beta/2$, $b=\beta(N-1)/2$, and $d=\beta(n-r)/2$. 
If $r=n$, the expected value follows along the lines, and the deterministic bound follows trivially from the boundedness of the beta distribution.
\end{proof}

\begin{remark}
In the context of random inputs, it is common to study the logarithm of a condition number instead of the condition number itself~\cite{BC2013,Edelman1988}. Thus, even when the expected condition is not finite, the expected logarithm may still be small. Using a standard argument (see, e.g., \cite[Proposition 2.26]{BC2013}) we can deduce a bound on the expected logarithm of the directional sensitivity:
\begin{equation*}
 \Expect[\log \sigma_E] \leq \log\left(\frac{1}{\gamma_P}\frac{2}{\pi}\sqrt{\frac{n-r}{N-1}}\right)+1.
\end{equation*}
The logarithm of the sensitivity is relevant as a measure for the loss of precision.
\end{remark}

As the derivation of the bounds~\eqref{eq:tailbounds} using Lemma~\ref{le:betaexp} shows, the cumulative distribution functions in question can be expressed {\em exactly} in terms of integrals of hypergeometric functions. This way, the tail probabilities can be computed to high accuracy for any given $t$, see also Remark~\ref{rem:spherical-gaussian}. 
However, as the derivation of the tail bounds in Appendix~\ref{sub:betaproof} also shows, the bounds given in Corollary~\ref{cor:expected} are sharp for fixed $t$ and $n-r\to \infty$, as well as for fixed $n-r$ and $t\to \infty$. Figure~\ref{fig:crude-tails} illustrates these bounds for a choice of small parameters ($n=4$, $d=2$, $r=2$, $\gamma_P=1$). 
Moreover, the bounds~\eqref{eq:tailbounds} have the added benefit of being easily interpretable. These tail bounds can be interpreted as saying that for large $n$ and/or $d$, it is highly unlikely that the directional sensitivity will exceed $\gamma_P^{-1}$ (which by Proposition~\ref{prop:cond-regular} is  the worst-case condition bound in the smooth case $r=n$). 

\begin{figure}[hbt]
\centering
\includegraphics[width=1\textwidth]{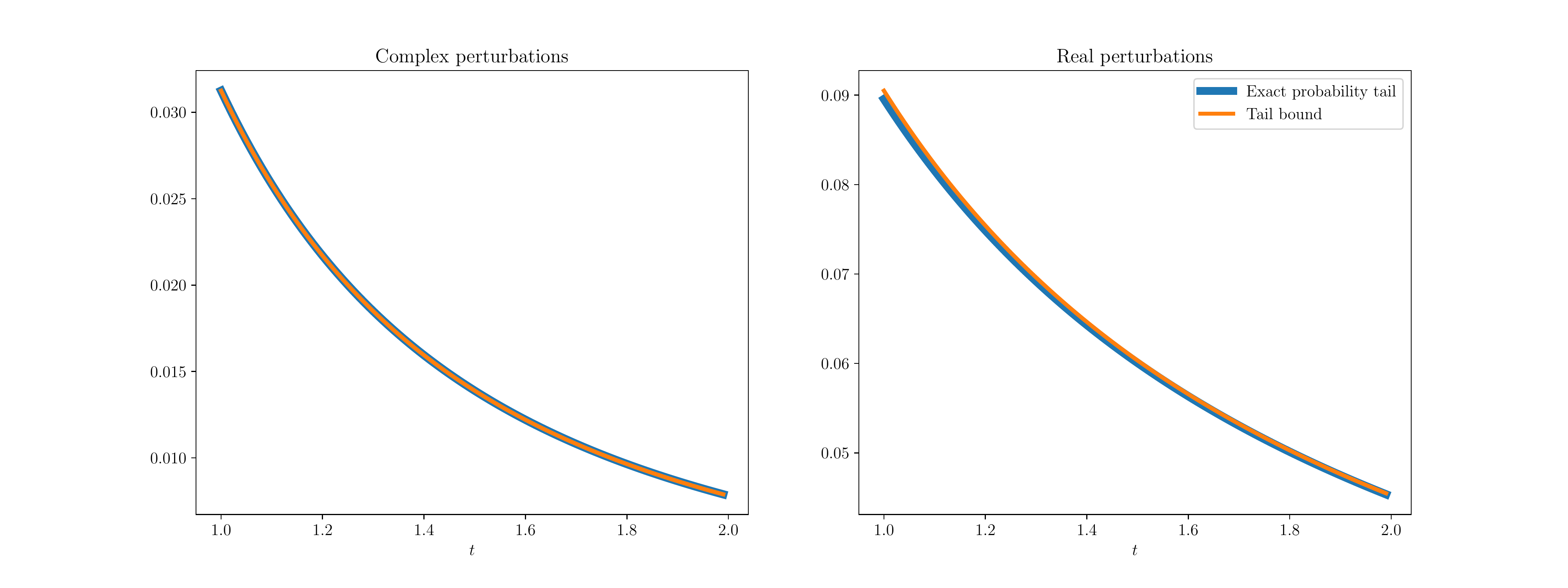}
\caption{The exact distribution tail of $\sigma_E$ and the tail bounds for $n=4$, $d=2$, $r=2$, and $\gamma_P=1$, so that $N=48$.}\label{fig:crude-tails}
\end{figure}

\begin{example}\label{ex:L-calculated}
Consider again the matrix pencil $L(x)$ from~\eqref{eq:pencil-example}. This pencil has rank $3$, and the cokernel and kernel are spanned by the vectors $p(x)$ and $q(x)$, respectively, given by
\begin{equation*}
  p(x) = \begin{bmatrix}
   1\\
  0\\
 -1\\
  1
  \end{bmatrix}, \quad
  q(x) = \begin{bmatrix}
x\\
-2x^2-4x+1\\
x\\
1
  \end{bmatrix}.
\end{equation*}
The matrix polynomial has the simple eigenvalue $\lambda=1$, and the matrix $L(1)$ has rank $2$. 
The cokernel $\ker L(1)^T$ and the kernel $\ker L(1)$ are spanned by the columns of the matrices $X$ and $Y$, given by
\begin{equation*}
 X = \begin{bmatrix}
 0.5774  & -0.7061\\
 0.          & 0.4888 \\
 -0.5774 & -0.4345\\
 0.5774  & 0.2716
\end{bmatrix}, \quad
 Y = \begin{bmatrix}
  0.1924 &-0.6873\\
 -0.9623 & -0.1322\\
  0.1924 & 0.02644\\
  0         & 0.7137
 \end{bmatrix}.
\end{equation*}
Let $u$ be the second column of $X$ and let $v$ be the second coumn of $Y$. The vectors $u$ and $v$ are orthogonal to $\ker_{\lambda} L(x)^T = \mathrm{span}\{p(1)\}$ and $\ker_{\lambda} L(x)=\mathrm{span}\{q(1)\}$ and have unit norm. We therefore have
\begin{equation*}
  \gamma_L := \frac{|u^TL'(1)v|}{\sqrt{2}} = 0.08223
\end{equation*}
Hence, $\gamma_L^{-1}=12.16$. Figure~\ref{fig:L-example} shows the result of comparing the distribution of $\sigma_E$, found empirically, with the bounds obtained in Theorem~\ref{thm:main-distr}. The relative error in the plot is of order $10^{-5}$. 

\begin{figure}[hbt]
\centering
\includegraphics[width=0.6\textwidth]{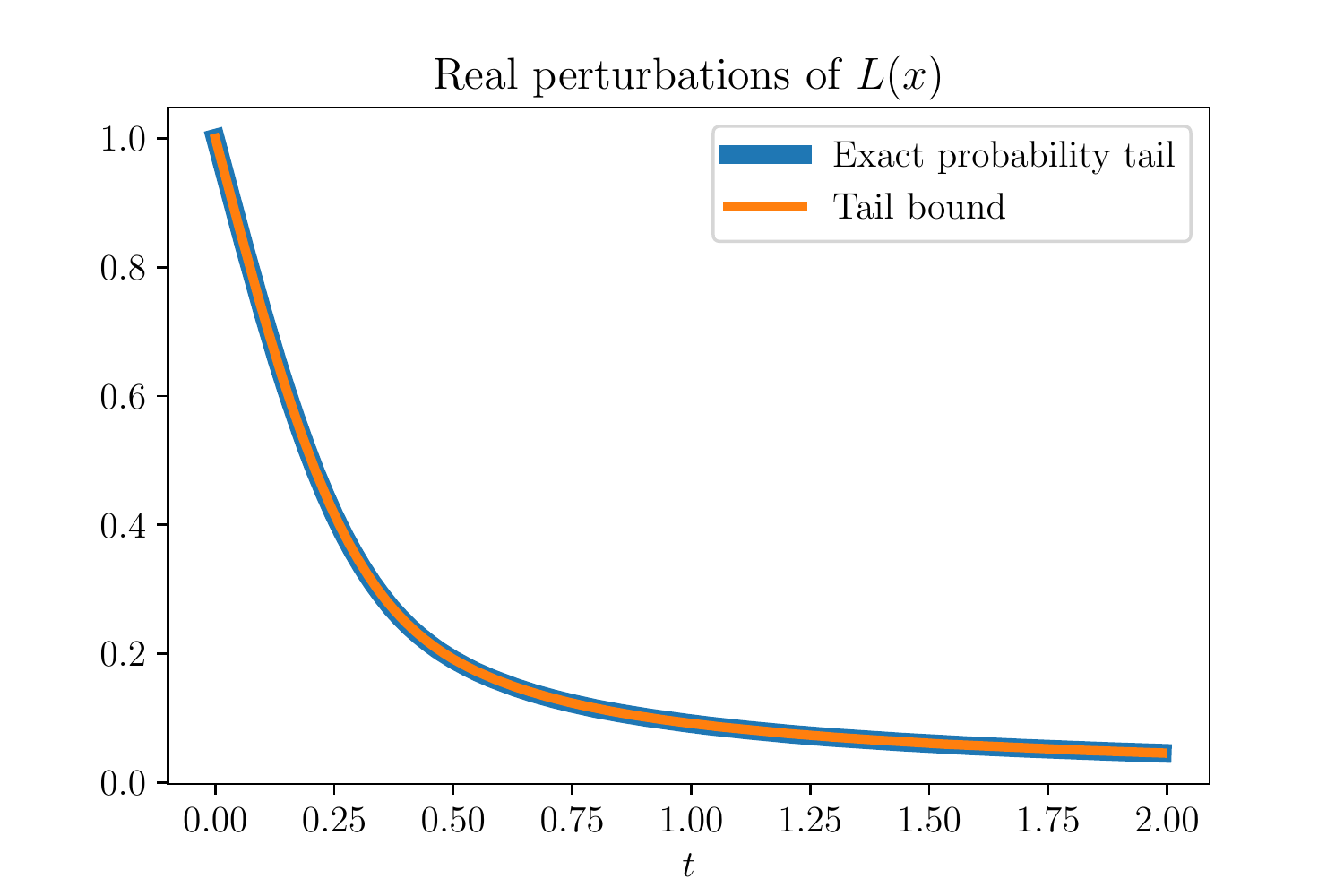}
\caption{The exact distribution tail of $\sigma_E$ for the matrix pencil $L(x)$ from~\eqref{eq:pencil-example}, and the theoretically computed tail bound~\eqref{eq:tailbounds}.}\label{fig:L-example}
\end{figure}
\end{example}

The plan for the rest of this section is as follows. In Section~\ref{sub:probprem} we recall some facts from probability theory and random matrix theory. In Section~\ref{sub:probqr} we discuss the QR decomposition of a random matrix, and in Section~\ref{sub:mainproof} we use this decomposition to prove Theorem~\ref{thm:main-distr}

\subsection{Probabilistic preliminaries}\label{sub:probprem}
We write $g\sim \normal^1(\mu,\Sigma)$ for a normal distributed (Gaussian) random vector $g$ with mean $\mu$ and covariance matrix $\Sigma$, and 
$g\sim \normal^2(\mu,\Sigma)$ for a complex Gaussian vector; this is a $\C^n$-valued random vector with expected value $\mu$, whose real and imaginary parts are independent real Gaussian random vectors with covariance matrix $\Sigma/2$ (a special case are real and complex scalar random variables, $\normal^{\beta}(\mu,\sigma^2)$).
We denote the uniform distribution on a sphere $S^{n-1}$ by $\uniform(n)$. Every Gaussian vector $g\sim \normal^1(0,I_n)$ can be written as a product $g=rq$ with $r$ and $q$ independent, where $r\sim \chi(n)$ is $\chi$-distributed with $n$ degrees of freedom, and $q\sim \uniform(n)$. 

\subsubsection{Projections of random vectors}\label{sub:projections}
The squared projected lengths of Gaussian and uniform distributed random vectors can be described using the $\chi^2$ and the beta distribution, respectively. A vector $X$ is $\chi^2$-distributed with $k$ degrees of freedom, $X\sim \chi^2(k)$, if the cumulative distribution function (cdf) is
\begin{equation*}
  \Prob\{X\leq x\} = \frac{1}{2^{k/2}\Gamma(k/2)}\int_0^x t^{k/2-1}e^{-t/2} \ \diff{t}.
\end{equation*}
The special case $\chi^2(2)$ is the exponential distribution with parameter $1/2$, written $\mathrm{exp}(1/2)$. The beta distribution $\Beta(a,b)$ is defined for $a,b>-1$, and has cdf supported on $[0,1]$,
\begin{equation}\label{eq:betadist}
 \Prob\{X\leq x\} = \frac{1}{\Beta(a,b)} \int_0^x t^{a-1}(1-t)^{b-1} \ \diff{t},
\end{equation}
where $\Beta(a,b)=\Gamma(a)\Gamma(b)/\Gamma(a+b)$ is the beta function.
For a vector $x\in \F^{n}$, denote by $\pi_k(x)$ the projection onto the first $k$ coordiantes and by  $\|\pi_k(x)\|^2=|x_1|^2+\dots+|x_k|^2$ its squared length. The following facts are known:
\begin{itemize}
\item If $g\sim \normal^\beta(0,I_n)$, then $\beta \|\pi_k(g)\|^2\sim \chi^2(\beta k)$;
\item If $q\sim \mathcal{U}(n)$, then $\|\pi_k(q)\|^2\sim \Beta(k/2,(n-k)/2)$.
\end{itemize}

The first claim is a standard fact about the normal distribution and can be derived directly from it, see for example~\cite{Boltzmann1881}. The statement for the uniform distribution can be derived from the Gaussian one, but also follows by a change of variables from expressions for the volume of tubular neighbourhoods of subspheres of a sphere, see for example~\cite[Section 20.2]{BC2013}.
Since all the distributions considered are orthogonally (in the real case) or unitarily (in the complex case) invariant, these observations hold for the projection of a random vector onto {\em any} $k$-dimensional subspace, not just the first $k$ coordinates.

\subsubsection{Random matrix ensembles}\label{subsec:randmat}
If $P(x)$ is a singular matrix polynomial with a simple eigenvalue $\lambda$, then the set of perturbation directions for which the directional sensitivity is not finite is a proper Zariski closed subset, see Theorem~\ref{thm:dd-2}. 
It is therefore natural and convenient to consider probability measures on the space of perturbations that have measure zero on proper Zariski closed subsets. This is the case, for example, if the measure is
absolutely continuous with respect to the Lebesgue measure. In this paper we will work with real and complex Gaussian and uniform distributions.
For a detailed discussion of the random matrix ensembles used here we refer to~\cite[Chapters 1-2]{forrester2010log}.

For a random matrix we write $G\sim \ginibre^\beta_n(\mu,\sigma^2)$ if each entry of $G$ is an independent $\normal^\beta(\mu,\sigma^2)$ random variable, and call this a Gaussian random matrix. In the case $\beta=2$ this is called the {\em Ginibre ensemble}~\cite{ginibre1965statistical}. Centered ($\mu=0$) Gaussian random matrices are orthogonally (if $\beta=1$) or unitarily (if $\beta=2$) invariant (\cite[Lemma 1]{mezzadri2007}) and the joint density of their entries is given by
\begin{equation*}
  \frac{1}{(2\pi/\beta)^{\beta n^2/2}} e^{-\frac{\beta \|G\|^2}{2}},
\end{equation*}
which takes into account the fact the real and imaginary parts of the entries of a complex Gaussian have variance $1/2$.
In addition, we
consider the {\em circular real ensemble} $\CRE(n)$ for real orthogonal matrices in $O(n)$, and the {\em circular unitary ensemble} $\CUE(n)$~\cite{dyson1962threefold} for unitary matrices in $U(n)$, where both distributions correspond to the unique Haar probability measure on the corresponding groups.

\subsection{The probabilistic QR decomposition}\label{sub:probqr}
Any nonsingular matrix $A\in \F^{n\times n}$ has a unique QR-decomposition $A=QR$,
where $Q\in O(n)$ (if $\F=\R$) or $U(n)$ (if $\F=\C$), and $R\in \F^{n\times n}$ is
upper triangular with $r_{ii}>0$~\cite[Part II]{trefethen97}. The following proposition describes the distribution of the factors $Q$ and $R$ in the QR-decomposition of a (real or complex) Gaussian random matrix. 

\begin{proposition}\label{prop:qr}
Let $G\sim \ginibre_n^\beta(0,1)$ be a Gaussian random matrix, $\beta\in \{1,2\}$. Then $G$ can be factored uniquely as $G=QR$, where $R=(r_{jk})_{1\leq j\leq k\leq n}$ is upper triangular and
\begin{itemize}
\item $Q\sim \CUE(n)$ if $\beta=2$ and $Q\sim \CRE(n)$ if $\beta=1$;
\item $\beta r_{ii}^2\sim \chi^2(\beta(n-i+1))$ for $i\in \{1,\dots,n\}$;
\item $r_{jk}\sim \normal^\beta(0,1)$ for $1\leq j<k\leq n$. 
\end{itemize}
Moreover, all these random variables are independent. 
\end{proposition}

An easy and conceptual derivation of the distribution of $Q$ can be found in~\cite{mezzadri2007}, while the distribution of $R$ can be deduced from the known expression for the Jacobian of the QR-decomposition~\cite[3.3]{edelman2005random}. 

\subsection{Proof of Theorem~\ref{thm:main-distr}}\label{sub:mainproof}
In this section we present the proofs of Theorem~\ref{thm:main-distr} and the corollaries that follow from it.
To simplify notation, we set $\ell=n-r+1$.
Recall from Corollary~\ref{cor:directional} the expression
\begin{equation*}
  \sigma_E = \frac{1}{\|E\|}\left|\frac{\det\left(X^*E(\lambda)Y\right)}{u^*P'(\lambda)v \cdot \det(U^*E(\lambda)V)}\right|,
\end{equation*}
where the columns of $X=[U\ u],Y=[V\ v]\in \F^{n\times \ell}$ are orthonormal bases of $\ker P(\lambda)^*$ and $\ker P(\lambda)$, the columns of $U,V$ represent bases of $\ker_\lambda P(x)^*$ and $\ker_\lambda P(x)$, respectively, and $\gamma_P$ is defined in~\eqref{eq:cdef}.

\par{\it Proof of Theorem~\ref{thm:main-distr}}. \ignorespaces
We first assume $r<n$. By the scale invariance of the directional sensitivity $\sigma_E$, we consider Gaussian perturbations $E\sim\normal^\beta(0,\sigma^2I_{\beta N})$ (recall that we interpret $E$ as a vector in $\F^N$), where $\sigma^2=(\sum_{j=0}^d |\lambda|^{2j})^{-1}$.  This scaling ensures that the entries of
$E(\lambda)$ are independent $\normal^{\beta}(0,1)$ random variables.
Since the distribution of $E(\lambda)$ is orthogonally/unitarily invariant, the quotient $|\det(X^*E(\lambda)Y|/|\det\left(U^*E(\lambda)V\right)|$ has the same distribution as the quotient $|\det(G)/\det(\overline{G})|$, where $G$ is the upper left $\ell\times \ell$ submatrix of $E(\lambda)$ and $\overline{G}$ the upper left $(\ell-1)\times (\ell-1)$ matrix. 
For the distribution considered, $G$ is almost surely invertible, with inverse $H=G^{-1}$. 
By Cramer's rule, $|\det(G)/\det(\overline{G})|=|h_{\ell\ell}|^{-1}$.
We are thus interested in the distribution
\begin{equation*}
  \Prob\{\sigma_E\geq t\} = \Prob\left\{ \left|\frac{\det(X^*E(\lambda)Y)}{\det(U^*E(\lambda)V)}\right| \geq |u^*P'(\lambda)v| t \|E\|\right\} = \Prob\left\{\frac{1}{|h_{\ell\ell}|} \geq |u^*P'(\lambda)v| t \|E\|\right\},
\end{equation*}
where $h_{\ell\ell}$ is the lower right corner of the inverse of an $\ell\times \ell$ Gaussian matrix $G$. To study the distribution of $|h_{\ell\ell}|^{-1}$, we resort to the probabilistic QR-decomposition discussed in Section~\ref{sub:probqr}. If $G=QR$ is the unique QR-decomposition of $G$ with positive diagonals in $R$, then the inverse is given by $H=R^{-1}Q^*$, and a direct inspection reveals that the lower-right element $h_{\ell\ell}$ of $H$ is $h_{\ell\ell} = q^*_{\ell\ell}/r_{\ell\ell}$. 

From Section~\ref{sub:probqr} it follows that $Q\sim \CRE(n)$ or $\CUE(n)$, and $\beta r_{\ell\ell}^2\sim \chi^2(\beta)$. Moreover, each column of $Q$ is uniformly distributed on the sphere $S^{\beta \ell-1}$, so that $|q_{\ell\ell}|^2\sim \Beta(\beta/2,\beta(\ell-1)/2)$ (by Section~\ref{sub:projections}), and $\{r_{\ell\ell}^2,|q_{\ell\ell}|^2\}$ are independent. We therefore get
\begin{equation*}
\Prob\left\{\frac{1}{|h_{\ell\ell}|} \geq |u^*P'(\lambda)v| t\|E\|\right\} = \Prob\left\{r_{\ell\ell}^2\geq |q_{\ell\ell}|^2|u^*P'(\lambda)v|^2t^2\|E\|^2 \right\}.
\end{equation*}
Setting $\gamma_P=|u^*P'(\lambda)v| \cdot (\sum_{j=0}^d |\lambda|^{2j})^{-1/2}$ (see~\eqref{eq:cdef}),
we arrive at
\begin{equation*}
\Prob\{\sigma_E\geq t\} = \Prob\left\{r_{\ell\ell}^2\geq |q_{\ell\ell}|^2\gamma_P^2t^2 \left(\sum_{j=0}^d |\lambda|^{2j}\right)\cdot \|E\|^2 \right\}
\end{equation*}
Let $p_0=(1,\lambda,\cdots,\lambda^d)^T/(\sum_{i=0}^d |\lambda|^{2i})^{1/2}$. Then
we can rearrange the coefficients of $E(x)$ to a matrix $F\in \F^{n^2\times (d+1)}$ so that
\begin{equation*}
  \|Fp_0\|^2 = \|E(\lambda)\|^2 \cdot \left(\sum_{i=0}^d |\lambda|^{2i}\right)^{-1}.
\end{equation*}
Moreover, if $Q=[p_0\ p_1\ \cdots p_{d}]$ is an orthogonal/unitary matrix with $p_0$ as first column, then 
\begin{equation*}
 \|E\|^2 = \|F\|^2 = \|FQ\|^2 = \left(\sum_{i=0}^d |\lambda|^{2i}\right)^{-1} \left(\|E(\lambda)\|^2 +\sum_{j=1}^d \|F\tilde{p}_j\|^2\right),
\end{equation*}
where $\tilde{p}_j=p_j\cdot (\sum_{i=0}^d |\lambda|^{2i})^{1/2}$. 
If we denote by $G^{c}$ the vector consisting of those entries of $E(\lambda)$ that are not in $G$, then
\begin{equation*}
 \|E(\lambda)\|^2 = \|G\|^2+\|G^{c}\|^2 = \|R\|^2+\|G^{c}\|^2.
\end{equation*}
It follows that 
\begin{equation*}
\left(\sum_{j=0}^d |\lambda|^{2j}\right)\cdot \|E\|^2 = \|R\|^2+\|G^{c}\|^2 + \sum_{j=1}^d \|F\tilde{p}_j\|^2. 
\end{equation*}
Therefore, the factor $r_{\ell\ell}^2$, itself a square of a (real or complex) Gaussian, is a summand in a sum of squares of $N=n^2(d+1)$ Gaussians, and the quotient
\begin{equation*}
\frac{r_{\ell\ell}^2}{\left(\sum_{j=0}^d |\lambda|^{2j}\right)\|E\|^2}
\end{equation*}
is equal to the squared length of the projection of a uniform random vector in $S^{\beta N-1}$ onto the first $\beta$ coordinates. By Section~\ref{sub:projections}, this is $\Beta(\beta/2,\beta(N-1)/2)$ distributed. Denoting this random variable by $Z_N$ and $|q_{\ell\ell}|^2$ by $Z_\ell$, we obtain
\begin{equation*}
\Prob\{\sigma_E\geq t\} = \Prob\{Z_N\geq \gamma_P^2t^2 Z_\ell\}.
\end{equation*}
This establishes the claim in the case $r<n$. 
If $r=n$, we use the expression (see~\eqref{eq:fran}),
\begin{equation*}
  \sigma_E = \left|\frac{u^*E(\lambda)v}{\|E\| |u^*P'(\lambda)v|}\right|,
\end{equation*}
where $u$ and $v$ are eigenvectors. By orthogonal/unitary invariance,
$\sigma_E^2$ has the same distribution as the the squared norm of a Gaussian. By the same argument as above, we can bound $\|E\|$ in terms of $\|E(\lambda)\|$, and
the quotient with $\|E(\lambda)\|^2$ is then the squared projected length of the first $\beta$ coordinates of a uniform distributed vector in $S^{\beta N-1}$, which is $\Beta(\beta/2,\beta(N-1)/2)$ distributed.
\endproof

\begin{remark}\label{rem:spherical-gaussian}
If $N+$ is large, then for a (real or complex) Gaussian perturbation with entry-wise variance $1/N$, by Gaussian concentration (see ~\cite[Theorem 5.6]{boucheron2013concentration}), $\|E\|$ is close to $1$ with high probability:
\begin{equation*}
  \Prob\{|\|E\|-1|\geq t\} \leq 2e^{-Nt^2/2}. 
\end{equation*}
This means that the distribution 
of $\|E\|\sigma_E$ for a Gaussian perturbation will be close to that of $\sigma_E$ for a uniform perturbation. 
Even for moderate sizes of $d$ and $n$, the result can be numerically almost indistinguishable.

In fact, when $G$ is Gaussian, then the distribution can be expressed explicitly as
\begin{equation*}
 \Prob\{|h_{\ell\ell}|^{-1}\geq t\} = {_1}F_1(1,\ell;-Nt^2),
\end{equation*}
where $_1F_1(a,b;z)$ denotes the confluent hypergeometric function (this follows by mimicking the proof of Theorem~\ref{thm:main-distr}, expressing the distribution in terms of a quotient of a $\chi^2$ and a beta random variable, and writing out the resulting integrals). Similarly, using the same computations as in the proof of Lemma~\ref{le:betaexp}, we get the exact expression
\begin{equation*}
 \Prob\{|h_{\ell\ell}|^{-1}\geq t\} = \begin{cases}
 1- {_2}F_1(1-N,1,\ell;t^2) & \text{ if } t\leq 1\\
 _2F_1(1-\ell,1,N;t^{-2}) & \text{ if } t\geq 1,
 \end{cases}
\end{equation*}
where $_2F_1(a,b,c;z)$ is the hypergeometric function.
The case distinction corresponds to different branches of the solution of the hypergeometric differential equation. See~\cite{pearson2009computation,pearson2017numerical} for more on computing with hypergeometric functions.
\end{remark}

\section{Weak condition numbers of simple eigenvalues of singular matrix polynomials}\label{sec:main}
The tail bounds on the directional sensitivity can easily be translated into statements about condition numbers and discuss some consequences and interpretations.

\begin{theorem}\label{thm:main-complex}
Let $P(x)\in \C^{n\times n}_d[x]$ be a matrix polynomial of rank $r$, and let $\lambda$ be a simple eigenvalue of $P(x)$. Then 
\begin{itemize}
\item the worst-case condition number is
\begin{equation}\label{eq:worst-cond-bound}
\kappa = \begin{cases} \infty & \text{ if } r<n,\\
\frac{1}{\gamma_P} & \text{ if } r=n
\end{cases}
\end{equation} 
\item the stochastic condition number, with respect to uniformly distributed perturbations, is
\begin{equation}\label{eq:stoch-cond-bound}
  \kappa_s = \frac{1}{\gamma_P}\frac{\pi}{2}\frac{\Gamma(N)\Gamma(n-r+1)}{\Gamma(N+1/2)\Gamma(n-r+1/2)}
\end{equation} 
\item if $r<n$ and $\delta\in (0,1)$, then the $\delta$-weak worst-case condition number, with respect to uniformly distributed perturbations, is bounded by
\begin{equation}\label{eq:weak-cond-bound}
  \kappa_{w}(\delta) \leq \frac{1}{\gamma_P}\max\left\{1,\sqrt{\frac{n-r}{\delta N}}\right\}
\end{equation}
\end{itemize}
\end{theorem}

The expression for the stochastic condition number involves the quotient of gamma functions, which can be simplified using the well-known bounds
\begin{equation}\label{eq:wallis}
  \sqrt{x} \leq \frac{\Gamma(x+1)}{\Gamma(x+1/2)} \leq \sqrt{x+1/2},
\end{equation}
which hold for $x>0$~\cite{wendel1948}. Using these bounds on the numerator and denominator of~\eqref{eq:stoch-cond-bound}, we get the more interpretable
\begin{equation}\label{eq:stoch-cond-bound-wallis}
\kappa_s \leq \frac{1}{\gamma_P} \frac{\pi}{2}\sqrt{\frac{n-r+1/2}{N-1/2}}\leq \frac{1}{\gamma_P} \frac{\pi}{2}\sqrt{\frac{n-r+1}{N}}.
\end{equation}
The bound on the weak condition number~\eqref{eq:weak-cond-bound} shows that $\kappa_w(1/2)$, which is the median of the same random variable of which $\kappa_s$ is the expected value, is bounded by $1/\gamma_P$, which is the expression of the worst-case condition number in the regular case $r=n$.

The situation changes dramatically when considering real matrix polynomials with real perturbations, as in this case even the stochastic condition becomes infinite if the matrix polynomial is singular. In the statement we denote the resulting condition number with respect to real perturbations by using the superscript $\R$.

\begin{theorem}\label{thm:main-real}
Let $P(x)\in \R^{n\times n}_d[x]$ be a real matrix polynomial of rank $r$, and let $\lambda\in \C$ be a simple eigenvalue of $P(x)$. Then 
\begin{itemize}
\item the worst-case condition number is
\begin{equation}\label{eq:worst-cond-bound-real}
  \kappa^{\R} = \begin{cases}
  \infty & \text{ if } r<n,\\
  \frac{1}{\gamma_P} & \text{ if } r=n
  \end{cases}
\end{equation}
\item the stochastic condition number, with respect to uniformly distributed real perturbations, is
\begin{equation}\label{eq:stoch-cond-bound-real}
  \kappa_s^{\R} = \begin{cases}
    \infty & \text{ if } r<n,\\
    \frac{1}{\gamma_P} \frac{\Gamma(N/2)}{\sqrt{\pi}\Gamma((N+1)/2)} &  \text{ if } r=n
  \end{cases}
\end{equation}
\item if $r<n$ and $\delta\in (0,1)$, then the $\delta$-weak worst-case condition, with respect to uniformly distributed real perturbations, is
\begin{equation}\label{eq:weak-cond-bound-real}
  \kappa_{w}^{\R}(\delta) \leq \frac{1}{\gamma_P}\max\left\{1,\sqrt{\frac{n-r}{N}} \frac{1}{\delta}\right\}
\end{equation}
\item if $r<n$ and $\delta < \sqrt{(n-r)/N}$, then the $\delta$-weak stochastic condition number satisfies
\begin{equation}\label{eq:weak-stoch-cond-bound-real}
\kappa_{ws}^{\R}(\delta) \leq \frac{1}{\gamma_P}\left(\frac{1}{1-\delta}\right)\left(1+\sqrt{\frac{n-r}{N}}\log\left(\sqrt{\frac{n-r}{N}}\delta^{-1}\right)\right)
\end{equation}
\end{itemize}
\end{theorem}

It is instructive to compare the weak condition numbers in the singular case to the worst-case and stochastic condition number in the regular case. In the regular case ($n=r$), when replacing the worst-case with the stochastic condition we get an improvement by a factor of $\approx N^{-1/2}$, which is consistent with previous work~\cite{armentano2010stochastic} (see also Section~\ref{sec:strong}) relating the worst-case to the stochastic condition. We will see in Section~\ref{sub:measure-concentration} that the expected value in the case $n=r$ captures the typical perturbation behaviour of the problem more accurately than the worst-case bound.
 Among the many possible interpretations of the weak worst-case condition, we highlight the following:
\begin{itemize}
\item Since the bounds are monotonically decreasing as the rank $r$ increases, we can get bounds independent of $r$. Specifically, we can replace the quotient $\sqrt{(n-r)/N}$ with $1/\sqrt{n(d+1)}$. This is useful since, in applications, the rank is not always known.
 \item While the stochastic condition number~\eqref{eq:stoch-cond-bound-real}, which measures the {\em expected} sensitivity of the problem of computing a singular eigenvalue, is infinite, for $4(n-r)<N$ the {\em median} sensitivity is bounded by 
\begin{equation*}
  \kappa^{\R}_w(1/2) \leq \frac{1}{\gamma_P}.
\end{equation*}
The median is a more robust and arguably better summary parameter than the expectation.
\item Choosing $\delta = e^{-N}$ in~\eqref{eq:weak-stoch-cond-bound-real}, we get a weak stochastic condition bound of
\begin{equation*}
\kappa^{\R}_{ws}(e^{-N}) \leq \frac{1}{\gamma_P} \left(\frac{1+\sqrt{N(n-r)}}{1-e^{-N}}\right).
\end{equation*}
That is, the condition number improves from being unbounded to sublinear in $N$, by just removing a set of inputs of exponentially small measure.  
\end{itemize}

\begin{example}\label{ex:L-pencil-ex1}
Consider the matrix pencil $L(x)$ from~\eqref{eq:pencil-example}. 
This matrix pencil has rank $3$, with only one simple eigenvalue $\lambda=1$. 
As we will see in Example~\ref{ex:L-calculated},
the constant $\gamma_L$ appearing in the bounds is
\begin{equation*}
  \gamma_L^{-1} = 12.16.
\end{equation*}
In this example, $n=4$, $d=1$ and $r=3$, so that $n-r=1$, $N=n^2(d+1)=32$, and 
\begin{equation*}
 \sqrt{\frac{n-r}{N}} = 0.1767. 
\end{equation*}
For small enough $\delta$ we get the (not optimized) bound
$\kappa_w^{\R}(\delta) <  2.15 \cdot \delta^{-1}$.
\end{example}

It is easy to translate Corollary~\ref{cor:expected} into the main results, Theorem~\ref{thm:main-complex} and Theorem~\ref{thm:main-real}. For the weak stochastic condition, we need the following observation, which is a variation of~\cite[Lemma 2.2]{amelunxen2017average}.

\begin{lemma}\label{le:waccolemma}
Let $Z$ be a random variable such that $\Prob\{Z\geq t\}\leq C\frac{a}{t}$ for $t>a$. Then for any $t_0>a$,
\begin{equation*}
\Expect[Z \ | \ Z\leq t_0] \leq \frac{a}{1-\frac{Ca}{t_0}}\left(1-C\log\left(\frac{a}{t_0}\right)\right).
\end{equation*}
\end{lemma}

\par{\it Proof of Theorem~\ref{thm:main-complex} and Theorem~\ref{thm:main-real}}. \ignorespaces
The statements about the worst-case,~\eqref{eq:worst-cond-bound} and~\eqref{eq:worst-cond-bound-real}, and about the stochastic condition number,~\eqref{eq:stoch-cond-bound} and~\eqref{eq:stoch-cond-bound-real}, follow immediately from Theorem~\ref{thm:main-distr} and Corollary~\ref{cor:expected}.

For the weak condition number in the complex case, if $\delta \leq (n-r)/N$, then setting 
\begin{equation*}
  t := \frac{1}{\gamma_P}\sqrt{\frac{n-r}{\delta N}}
\end{equation*}
we get $\gamma_P t \geq 1$, and therefore, using the complex tail bound from Corollary~\ref{cor:expected},
\begin{equation*}
  \Prob\{\sigma_E \geq t\} \leq \frac{1}{\gamma_P^2t^2}\frac{n-r}{N} = \delta.
\end{equation*}
This yields $\kappa_w(\delta)\leq t$. If $\delta > (n-r)/N$, then we use the fact that the weak condition number is monotonically decreasing with $\delta$ (intuitively, the larger the set we are allowed to exclude, the smaller the condition number will be), to conclude that $\kappa_w(\delta)\leq \kappa_w(\delta_0)\leq 1/\gamma_P$, where $\delta_0:=(n-r)/N$. 

For the real case, if $r<n$ we use the bound
\begin{equation*}
  \frac{\Gamma(N/2)\Gamma((n-r+1)/2)}{\Gamma((N+1)/2)\Gamma((n-r)/2)} \leq \sqrt{\frac{n-r}{N-1}},
\end{equation*}
which follows from~\eqref{eq:wallis}. If $\delta<\sqrt{(n-r)/N}$, set
\begin{equation*}
  t := \frac{1}{\gamma_P}\sqrt{\frac{n-r}{N}}\frac{1}{\delta}.
\end{equation*}
Then
\begin{equation*}
 \Prob\{\sigma_E\geq t\} \leq \frac{1}{\gamma_P}\frac{2}{\pi}\sqrt{\frac{n-r}{N}}\frac{1}{t} = \frac{2}{\pi}\sqrt{\frac{N}{N-1}}\cdot \delta \leq \delta,
\end{equation*}
where for the last inequality we used the fact that $N\geq 2$. 
We conclude that $\kappa_w(\delta)\leq t$. If $\delta>\sqrt{(n-r)/N}$, then we use the monotonicity of the weak condition just as in the complex case.
Finally, for the weak stochastic condition number in the real case, we use Lemma~\ref{le:waccolemma} with $a=\gamma_P^{-1}$, $C=\sqrt{(n-r)/N}$ and $t_0=C(\delta \gamma_P)^{-1}$ in the conditional expectation. We just saw that $\kappa_{w}^{\R}(\delta)\leq t_0$, so that
\begin{equation*}
  \kappa_{ws}^{\R}(\delta) \leq \Expect[\sigma_E \ | \ \sigma_{E}\leq t_0] \leq \frac{1}{\gamma_P}\left(\frac{1}{1-\delta}\right)\left(1+\sqrt{\frac{n-r}{N}}\log\left(\sqrt{\frac{n-r}{N}}\delta^{-1}\right)\right),
\end{equation*}
where we used Lemma~\ref{le:waccolemma} in the second inequality.
\endproof

\section{Bounding the weak stochastic condition number}\label{sec:compweak}
In this section we illustrate how the weak condition number of the problem of computing a simple eigenvalue of a singular matrix polynomial can be estimated in practice. More precisely, we show that the weak condition number of a singular problem can be estimated in terms of the stochastic condition number of nearby regular problems. Before deriving the relevant estimates, given in Theorem~\ref{thm:main-approx}, we discuss the stochastic condition number of regular matrix polynomials.

\subsection{Measure concentration for the directional sensitivity of regular matrix polynomials}\label{sub:measure-concentration}

For the directional sensitivity in the regular case, $r=n$, the worst-case condition number is $\gamma_P^{-1}$, as was shown in Proposition~\ref{prop:cond-regular}. In addition, the expression for the stochastic condition number involves a ratio of gamma functions (see Corollary~\ref{cor:expected} or the case $r=n$ in Theorem~\ref{thm:main-complex} and Theorem~\ref{thm:main-real}). From~\eqref{eq:wallis} we get the approximation $\Gamma(k+1/2)/\Gamma(k)\approx \sqrt{k}$, so that the stochastic condition number for regular polynomial eigenvalue problems satisfies
\begin{equation*}
  \kappa_s \approx \frac{1}{\sqrt{N}} \kappa.
\end{equation*}
This is compatible with previously known results about the stochastic condition number in the smooth setting (see the discussion in Section~\ref{sec:strong}). A natural question is whether the directional sensitivity is likely to be closer to this expected value, or closer to the upper bound $\kappa$.

Theorem~\ref{thm:main-distr} describes the distribution of $\sigma_E$ as that of the (scaled) square root of a beta random variable. 
Using the interpretation of beta random variables as squared lengths of projections of uniformly distributed vectors on the sphere (see Section~\ref{sub:projections}), tail bounds for the distribution of $\sigma_E$ therefore translate into the problem of bounding the relative volume of certain subsets of the unit sphere. A standard argument from the realm of measure concentration on spheres, Lemma~\ref{le:concentration}, then implies that with high probability, $\sigma_E$ will stay close to its mean. 

\begin{lemma}\label{le:concentration}
Let $x\sim \uniform(\beta N)$ be a uniformly distributed vector on the (real or complex) unit sphere, where $\beta=1$ if $\F=\R$ and $\beta=2$ if $\F=\C$. Then 
\begin{equation*}
\Prob\{|x_1| \geq t\} \leq e^{-\beta (N-1)t^2/2}.
\end{equation*}
\end{lemma}

\begin{proof}
For complex perturbations, we get the straight-forward bound
\begin{equation*}
 \Prob\{|x_1|\geq t\} \leq (1-t^2\gamma_P^2)^{N-1} \leq e^{-(N-1)t^2}.
\end{equation*}
In the real case, a classic result (see~\cite[Lemma 2.2]{ball1997elementary} for a short and elegant proof) states that the probability in question is bounded by
\begin{equation}\label{eq:beta-tail-bound}
  \Prob\{|x_1|\geq t\} \leq e^{-Nt^2/2}.
\end{equation}
The claimed bound follows by replacing $N$ with $N-1$ for the sake of a uniform presentation.
\end{proof}
 
The next corollary follows from the description of the distribution of $\sigma_E$ in Theorem~\ref{thm:main-distr}, and the characterization of beta random variables as squared projected lengths of uniform vectors from Section~\ref{sub:projections}.
 
\begin{corollary}
Let $P(x)\in \F_d^{n\times n}[x]$ be a regular matrix polynomial and let $\lambda$ be a simple eigenvalue of $P(x)$. If $E\sim \uniform(\beta N)$, where $\beta=1$ if $\F=\R$ and $\beta=2$ if $\F=\C$, then for $t\leq \gamma_P^{-1}$ we have
\begin{equation*}
\Prob\{\sigma_E \geq t\} \leq e^{-\beta (N-1)\gamma_P^2 t^2/2}.
\end{equation*}
\end{corollary}

\subsection{The weak condition number in terms of nearby stochastic condition numbers}\label{sub:weak-in-term-of-stochastic}

It is common wisdom that computing the condition number is as hard as solving the problem at hand, so at the very least we would like to avoid making the computation of the condition estimate more expensive than the computation of the eigenvalue itself. We will therefore aim to estimate the condition number of the problem in terms of the output of a backward stable algorithm for computing the eigenvalue and a pair of associated eigenvectors. 

Let $P(x)\in \F^{n\times n}_d[x]$ be a matrix polynomial of rank $r<n$ with a simple eigenvalue $\lambda\in \C$, and let $E(x)\in \F^{n\times n}_d[x]$ be a regular perturbation. Denote by $\lambda(\epsilon)$ the eigenvalue of $P(x)+\epsilon E(x)$ that converges to $\lambda$ (see Theorem~\ref{thm:dd-2}), and let $u(\epsilon)$ and $v(\epsilon)$ be the corresponding left and right eigenvectors of the perturbed problem.  As shown in~\cite[Theorem 4]{dd10} (see Theorem~\ref{thm:dd-eigenvector} below), for all $E(x)$ outside a proper Zariski closed set, the limits
\begin{equation}\label{eq:limit-ev}
  \overline{u}= \lim_{\epsilon \to 0} u(\epsilon), \quad \quad \overline{v} = \lim_{\epsilon \to 0} v(\epsilon)
\end{equation}
converge to representatives of left and right eigenvectors of $P(x)$ associated to $\lambda$. Whenever these limits exist and represent eigenvectors of $P(x)$, define
\begin{equation}\label{eq:approx-gamma}
 \overline{\gamma}_P := \overline{u}^*P'(\lambda)\overline{v}\cdot \bigg(\sum_{j=0}^d |\lambda|^{2j}\bigg)^{-1/2}, \quad \quad \overline{\kappa} = \overline{\gamma}_{P}^{-1}, \quad \text{ and } \quad \overline{\kappa}_s =  \overline{\gamma}_P^{-1}  \begin{cases}
\frac{\sqrt{\pi}}{2} \frac{\Gamma(N)}{\Gamma(N+1/2)} & \text{ if } \F=\C,\\
\frac{1}{\sqrt{\pi}} \frac{\Gamma(N/2)}{\Gamma((N+1)/2)} & \text{ if } \F=\R.
 \end{cases}
\end{equation}
Note that these parameters depend implicitly on a perturbation direction $E(x)$, even though the notation does not reflect this. 
The parameters $\overline{\kappa}$ and $\overline{\kappa}_s$ are the limits of the worst-case and stochastic condition numbers, $\kappa(P(x)+\epsilon E(x))$ and $\kappa_s(P(x)+\epsilon E(x))$, as $\epsilon \to 0$. 
Since almost sure convergence implies convergence in probability, we get
\begin{equation}\label{eq:asymptotic-smoothed}
  \Expect[\overline{\kappa}_s] = \lim_{\epsilon \to 0} \Expect[\kappa_s(P(x)+\epsilon E(x))]
\end{equation}
whenever the left-hand side of this expression is finite. 

A backward stable algorithm, such as vanilla QZ, computes an eigenvalue $\tilde{\lambda}$ and associated unit-norm eigenvectors $\tilde{u}$ and $\tilde{v}$ of a nearby problem $P(x)+\epsilon E(x)$. If $\epsilon$ is small, then $\tilde{\lambda}\approx \lambda$, $\tilde{u}\approx \overline{u}$ and $\tilde{v}\approx \overline{v}$, so that we can approximate the values~\eqref{eq:approx-gamma} using the output of such an algorithm. Unfortunately, this does not yet give us a good estimate of $\gamma_P$, as the definition of $\gamma_P$ makes use of very special representatives of eigenvectors (recall from Section~\ref{sub:ev} that for a singular matrix polynomials, eigenvectors are only defined as equivalence classes). The following theorem shows that we can still get bounds on the weak condition numbers in terms of $\overline{\kappa}_s$. 

\begin{theorem}\label{thm:main-approx}
Let $P(x)\in \F^{n\times n}_d[x]$ be a singular matrix polynomial of rank $r<n$ with simple eigenvalue $\lambda\in \C$. Then
\begin{equation*}
  \kappa_w(\delta) \leq \overline{\kappa}\cdot \max \left\{\delta^{-1/\beta}\sqrt{\frac{n-r}{N}},1\right\}.
\end{equation*}
If $\delta \leq (n-r)/N$, then for any $\eta>0$ we have the tail bounds
\begin{equation*}
\Prob\left\{\delta^{-1/\beta}\overline{\kappa}_s \geq \eta\cdot \kappa_w(\delta)\right\} \geq 1-e^{-\beta/\eta^2}.
\end{equation*}
\end{theorem}

For the proof of Theorem~\ref{thm:main-approx} we recall the setting of Section~\ref{sec:eigencond}.
Let $X=[U \ u]$ and $Y=[V \ v]$ be matrices whose columns are orthonormal bases of $\ker P(\lambda)^*$ and $\ker P(\lambda)$, respectively, and such that $U$ and $V$ are bases of $\ker_{\lambda}P(x)^{T}$ and $\ker_{\lambda}P(x)$, respectively. If $\overline{u}=u$ and $\overline{v}=v$ in~\eqref{eq:approx-gamma}, then $\overline{\gamma}_P=\gamma_P$. 
In general, however, we only get a bound.  
To see this, recall from Section~\ref{sub:ev} that $\overline{u}^*P'(\lambda)\overline{v}$ depends only on the component of $\overline{u}$ that is orthogonal to $\ker_{\lambda}P(\lambda)^*$, and the component of $\overline{v}$ that is orthogonal to $\ker_{\lambda}P(\lambda)$. In particular, $X^*P'(\lambda)Y$ has rank one, and
we have (recall $\ell=n-r+1$)
\begin{equation}\label{eq:rank-one-char}
  X^*P'(\lambda)Y = u^*P'(\lambda) v \cdot e_{\ell}e_{\ell}^*.
\end{equation}

The key to Proposition~\ref{thm:main-approx} lies in a result analogous to Theorem~\ref{thm:dd-2} for the eigenvectors by de Ter\'an and Dopico~\cite[Theorem 4]{dd10}.

\begin{theorem}\label{thm:dd-eigenvector}
Let $P(x)\in \F^{n\times n}_d[x]$ be matrix polynomial of rank $r$ with simple eigenvalue $\lambda$ and $X,Y$ as above. Let $E(x)\in \F^{n\times n}_d[x]$ be such that $X^*E(\lambda)Y$ is non-singular. Let $\zeta$ be the eigenvalue of the non-singular matrix pencil
\begin{equation}\label{eqn:pencil-to-solve}
 X^*E(\lambda)Y+ \zeta \cdot X^*P'(\lambda)Y,
\end{equation}
and let $a$ and $b$ be the corresponding left and right eigenvectors. 
Then for small enough $\epsilon>0$, the perturbed matrix polynomial $P(x)+\epsilon E(x)$ has exactly one eigenvalue $\lambda(\epsilon)$ as described in Theorem~\ref{thm:dd-2}, and the corresponding left and right eigenvectors satisfy
\begin{equation*}
u(\epsilon) = Xa+O(\epsilon), \quad \quad v(\epsilon) = Yb+O(\epsilon).
\end{equation*}
\end{theorem}

Given a matrix polynomial $P(x)$ and a perturbation direction $E(x)$, we can therefore assume that the eigenvectors of a sufficiently small perturbation in direction $E(x)$ are approximated by $\overline{u}=Xa$ and $\overline{v}=Yb$, where $a,b$ are the eigenvectors of the matrix pencil~\eqref{eqn:pencil-to-solve}. We would next like to characterize these eigenvectors for {\em random} perturbations $E(x)$. As with the rest of this paper, the following result is parametrized by a parameter $\beta\in \{1,2\}$ which specifies whether we work with real or complex perturbations.

\begin{proposition}\label{prop:prob-estimate}
Let $P(x)\in \F^{n\times n}_d[x]$ be a matrix polynomial of rank $r<n$ with simple eigenvalue $\lambda\in \C$, and let $E(x)\sim \uniform(\beta N)$ be a random perturbation. Let $a,b$ be left and right eigenvectors of the linear pencil~\eqref{eqn:pencil-to-solve}, let $\overline{u}=Xa$ and $\overline{v}=Yb$, and define $\overline{\gamma}_P$ as in~\eqref{eq:approx-gamma}. 
Then
\begin{equation*}
 \Expect[\overline{\gamma}_P] \leq (\ell-1)^{-1/2}\gamma_P, \quad \text{ and } \quad  \Prob\{ \overline{\gamma}_P\geq \gamma_P\cdot t\} \leq e^{-\beta (\ell-1) t^2/2}.
\end{equation*}
\end{proposition}

\begin{proof}
By scale invariance of~\eqref{eqn:pencil-to-solve}, we may take $E(x)$ to be Gaussian, $E(x)\sim \normal^{\beta}(0,\sigma^2I_{\beta N})$ with $\sigma^2=(\sum_{j=0}^d |\lambda|^{2j})^{-1}$ (so that $E(\lambda)\sim \ginibre_n^\beta(0,1)$).
Set $G:=X^*E(\lambda)Y$, so that $G\sim \ginibre_\ell^\beta(0,1)$. Using~\eqref{eq:rank-one-char}, the eigenvectors associated to~\eqref{eqn:pencil-to-solve} are then characterized as solutions of
\begin{equation*}
a^*(G+\zeta \cdot \gamma_P e_{\ell}e_{\ell}^*) = 0, \quad \quad (G+ \zeta \cdot \gamma_P e_{\ell}e_{\ell}^*)b = 0.
\end{equation*}
It follows that $G^*a$ and $Gb$ are proportional to $e_\ell$, and hence
\begin{equation*}
  a = \frac{G^{-*}e_{\ell}}{\|G^{-*}e_{\ell}\|}, \quad \quad b = \frac{G^{-1}e_{\ell}}{\|G^{-1}e_{\ell}\|}.
\end{equation*}
Clearly, each of the vectors $a$ and $b$ individually is uniformly distributed. They are, however, not independent. To simplify notation, set $H=G^{-1}$. For the condition estimate we get, using~\eqref{eq:rank-one-char},
\begin{equation*}
  |a^*X^*P'(\lambda)Yb| = |u^*P'(\lambda)v| \cdot \frac{| e_{\ell}^*H e_{\ell}| \cdot |e_{\ell}^*He_{\ell}|}{\|H^*e_{\ell}\|\cdot \|He_{\ell}\|} \leq |u^*P'(\lambda)v| \cdot \frac{|e_{\ell}^*H e_{\ell}|}{\|He_\ell\|}.
\end{equation*}
By orthogonal/unitary invariance of the Gaussian distribution, the random vector $q:=He_{\ell}/\|H e_{\ell}\|$ is uniformly distributed on $S^{\beta\ell-1}$. It follows that $|e^{T}He_{\ell}|/\|H e_{\ell}\|$ is distributed like the absolute value of the projection of a uniform vector onto the first coordinate. 
For the expected value, the bound follows by observing that the expected value of such a projection is bounded by $(\ell-1)^{-1/2}$. 
For the tail bound, using~\eqref{eq:beta-tail-bound} (with $N$ replaced by $\ell$) we get
\begin{equation*}
 \Prob\{ \overline{\gamma}_P \geq \gamma_P\cdot t\} = \Prob\{ |a^*X^*P'(\lambda)Yb| \geq |u^*P'(\lambda)v|\cdot t\}\leq e^{-\beta (\ell-1) t^2/2}.
\end{equation*}
This was to be shown.
\end{proof}

\proofof{Theorem~\ref{thm:main-approx}}
If $\overline{u}=Xa$ and $\overline{v}=Yb$, then 
\begin{equation*}
  |\overline{u}^*P'(\lambda)\overline{v}| = |a^{T}X^*P'(\lambda)Yb| = |a_{\ell}b_{\ell}u^*P'(\lambda)v| \leq |u^*P'(\lambda) v|,
\end{equation*}
and we get the upper bound 
\begin{equation*}
  \gamma_P^{-1} \leq \overline{\gamma}_P^{-1} = \overline{\kappa}.    
\end{equation*}
For the weak condition numbers, using Theorem~\ref{thm:main-complex} and Theorem~\ref{thm:main-real}, we get the bounds
\begin{equation*}
  \kappa_w(\delta) \leq \overline{\kappa}\cdot \max\left\{\sqrt{\frac{n-r}{N}}\frac{1}{\delta^{1/\beta}},1\right\}.
\end{equation*}

For the tails bounds in the complex case, note that in the complex case we have
\begin{align*}
\Prob\left\{(N\delta)^{-1/2}\overline{\kappa} \leq \eta\cdot \kappa_w(\delta)\right\} = \Prob\left\{\overline{\gamma}_P^{-1}\leq \eta\cdot \sqrt{n-r}\ \gamma_P^{-1}\right\}= \Prob\left\{\overline{\gamma}_P\geq \eta^{-1}\cdot  (n-r)^{-1/2}\gamma_P\right\} \leq e^{-1/\eta^2},
\end{align*}
where we used Proposition~\ref{prop:prob-estimate} for the inequality. The real case follows in the same way. 
\endproof

\section{Conclusions and outlook}\label{sec:conclusions}
The classical theory of conditioning in numerical analysis aims to quantify the susceptibility of a computational problem to perturbations in the input. While the theory serves its purpose well in distinguishing well-posed problems from problems that approach ill-posedness, it fails to explain why certain problems with high condition number can still be solved satisfactory to high precision by algorithms that are oblivious to the special structure of an input. By introducing the notions of weak and weak stochastic conditioning, we developed a tool to better quantify the perturbation behaviour of numerical computation problems for which the classical condition number fails to do so. 

Our methods are based on an analysis of directional perturbations and probabilistic tools. The use of probability theory in our context is auxiliary: the purpose is to quantify the observation that the set of adversarial perturbations is small. 
In practice, any reasonable numerical algorithm will find the eigenvalues of a nearby regular matrix polynomial, and the perturbation will be deterministic and not random. However, as the algorithm knows nothing about the particular input matrix polynomial, it is reasonable to assume that if the set of adversarial perturbations is sufficiently small, then the actual perturbation will not be in there. 
Put more directly, to say that the probability that a perturbed problem has large directional sensitivity is very small is to say that a perturbation, although non-random, would need a good reason to cause damage.

The results presented continue the line of work of~\cite{amelunxen2017average}, where it is argued that, just as sufficiently small numbers are considered numerically indistinguishable from zero, sets of sufficiently small measure should be considered numerically indistinguishable from null-sets. One interesting direction in which the results presented can be strengthened is to use wider classes of probability distributions, including such that are discrete, and derive equivalent (possibly slightly weaker) results. 
One important side-effect of our analysis is a focus away from the expected value, and more towards robust measures such as the median\footnote{The use of the median instead of the expected value in the probabilistic analysis of quantities was suggested by F.~Bornemann~\cite{bornemann}.} and other quantiles.

Our results hence have a couple of important implications, or ``take-home messages'', that we would like to highlight:
\begin{enumerate}
\item The results presented call for a critical re-evaluation of the notion of ill-posedness. It has become common practice to simply {\em identify} ill-posedness with having infinite condition, to the extent that condition numbers are often {\em defined} in terms of the inverse distance to a set of ill-posed inputs, an approach that has been popularized by J.~Demmel~\cite{Demmel1987,Demmel1988}.\footnote{For the complexity analysis of iterative algorithms, and in particular for problems related to convex optimization, the ``distance to ill-posedness'' approach may often be the most natural setting. For convex feasibility problems, for example, the ill-posed inputs form a wall separating primal from dual feasible problem instances, and closeness to this wall directly affects the speed of convergence of iterative algorithms; see~\cite{BC2013} for more on this story.} The question of whether the elements of such a set are actually badly behaved a practical sense is often left unquestioned. 
Our theory suggests that the set of inputs that are actually ill-behaved from a practical point of view can be 
smaller than previously thought.
\item Average-case analysis (and its refinement, smoothed analysis~\cite{burgisser2010smoothed}) is, while well-intentioned, still susceptible to the caprices of specific probability distributions. More meaningful results are obtained when, instead of analysing the behaviour of perturbations on average, one shifts the focus towards showing that the set of {\em adversarial} perturbations is small; ideally so small, that hitting a misbehaving perturbation would suggest the existence of a specific explanation for this rather than just bad luck. In terms of summary parameters, our approach suggests using, in line with common practice in statistics, more robust parameters such as the median instead of the mean.
\end{enumerate}

A natural question that arises from the first point is: if some problems that were previously thought of as ill-posed are not (in the sense that the set of discontinuous perturbation directions is negligible), then which problems are genuinely ill-posed? In the case of polynomial eigenvalue problems, we conjecture that problems with semisimple eigenvalues are not ill-conditioned in our framework; in fact, it appears that much of the analysis performed in this section can be extended to this setting. It is not completely obvious which problems should be considered ill-posed based on this new theory. 
That some inputs still should can be seen for example by considering Jordan blocks with zeros on the diagonal; the computed eigenvalues of perturbations of the order of machine precision will not recover the correct eigenvalue in this situation. Our analysis in the semisimple case is based on the fact that the directional derivative of the function to be computed exists in sufficiently many directions. 

Another consequence is that much of the probabilistic analyses of condition numbers based on the distance to ill-posedness, while still correct, can possibly be refined when using a smaller set of ill-posed inputs. In particular, it is likely that condition bounds resulting from average-case and smoothed analysis can be refined. Finally, an interesting direction would be to examine problems with high or infinite condition number that are not ill-posed in a practical sense in different contexts, such as polynomial system solving or problems arising from the discretization of continuous inverse problems.

\section*{Acknowledgements}
The spark that led to this paper was ignited at the workshop ``Algebra meets numerics: condition and complexity'' on November 6-7, 2017 in Berlin; we are grateful to the organisers Peter B\"{u}rgisser and Felipe Cucker for inviting us and for pointing out the work of Armentano and Stewart.
In addition the authors would like to thank Carlos Beltr\'an  and Daniel Kressner for valuable feedback, and the anonymous referees for useful comments.
We are greatly indebted to Dennis Amelunxen, whose vision of weak-average case analysis inspired this work. 

We would like to acknowledge financial support from the Manchester Institute for Mathematical Sciences (MIMS) during the early stages of this project, and the Isaac Newton Institute for Mathematical Sciences for support and hospitality during the programme Approximation, Sampling and Compression in Data Science while this work was completed.

\bibliographystyle{abbrv}
\bibliography{paper}

\appendix
\section{Moments and tails for ratios of beta random variables}\label{sub:betaproof}
In this appendix we compute the expected value and tail bounds for moments of quotients of beta random variables.

\begin{lemma}\label{le:betaexp}
Let $a>0,b>0,c>0,d>0$ and $X\sim \Beta(a,b)$, $Y\sim \Beta(c,d)$ random variables. Then for $k$ such that $ck>1$,
\begin{equation*}
  \Expect\left[X^{1/k}\right] = \frac{\Beta(a+1/k,b)}{\Beta(a,b)}\quad \text{ and } \quad \Expect\left[(X/Y)^{1/k}\right] = \frac{\Beta(a+1/k,b)\Beta(c-1/k,d)}{\Beta(a,b)\Beta(c,d)},
\end{equation*}
If $ck=1$, then $\Expect[(X/Y)^k]=\infty$. Moreover, the probability tails are bounded by
\begin{equation*}
  \Prob\{(X/Y)^{1/k} \geq t\} \leq 
  \begin{cases}
  1-t^{ak}\frac{\Beta(a+c,d)}{\Beta(c,d)} & \text{ if } t\leq 1\\
  \frac{1}{t^{ck}} \frac{\Beta(a+c,b)}{c \Beta(a,b)\Beta(c,d)} & \text{ if } t\geq 1
  \end{cases}
\end{equation*}
\end{lemma}

\begin{proof}
We focus on the case of the quotient $(X/Y)^{1/k}$; the statement for $X^{1/k}$ follows by simply setting $Y=1$ in the calculations below. Set $C=1/(\Beta(a,b)\Beta(c,d))$. For $ck>1$,
\begin{align*}
\Expect\left[\left(X/Y\right)^{1/k}\right] &= \int_{t=0}^\infty \Prob\{X/Y\geq t^k\} \ \diff{t}\\
&= C \cdot \int_{y=0}^1 \int_{t=0}^{y^{-1/k}} \int_{x=t^k y}^1 x^{a-1}(1-x)^{b-1}y^{c-1}(1-y)^{d-1} \ \diff{x} \ \diff{t} \ \diff{y}\nonumber \\
&\stackrel{s=y^{1/k}t}{=} C \cdot \int_{y=0}^1 \int_{s=0}^1 \int_{x=s^{k}}^1 x^{a-1}(1-x)^{b-1}y^{c-1/k-1}(1-y)^{d-1} \ \diff{x} \ \diff{s} \ \diff{y} \label{eq1}\tag{A}\\
&= C \cdot \Beta(c-1/k,d) 
\int_{s=0}^1 \int_{x=s^{k}}^1 x^{a-1}(1-x)^{b-1}\ \diff{x} \ \diff{s} \label{eq2}\tag{B}\\
&= C \cdot \Beta(c-1/k,d) \int_{x=0}^1 \int_{s=0}^{x^{1/k}} x^{a-1}(1-x)^{b-1} \ \diff{s} \ \diff{x} = C \cdot \Beta(a+1/k,b) \Beta(c-1/k,d). \nonumber 
\end{align*}
If $ck=1$, then the step from~\eqref{eq1} to~\eqref{eq2} breaks down, since the integral $\int_{0}^1 y^{-1}(1-y)^{d-1} \ \diff{y}$ diverges. For the tail bound we proceed similarly. If $t\leq 1$, then
\begin{align*}
\Prob\{(X/Y)^{1/k} \geq t\}  &= 1- C \cdot \int_{y=0}^{1} \int_{x=0}^{t^ky} x^{a-1}(1-x)^{b-1}y^{c-1}(1-y)^{d-1} \ \diff{x} \ \diff{y}\\
&\stackrel{x=(t^{k}y)z}{=} 1- C \cdot t^{ak} \int_{y=0}^1  \int_{z=0}^1 z^{a-1}(1-t^kyz)^{b-1}y^{a+c-1}(1-y)^{d-1} \ \diff{y} \ \diff{z}\\
& \stackrel{(1)}{\leq} 1- C \cdot t^{ak} \int_{y=0}^1  \int_{z=0}^1 z^{a-1}(1-z)^{b-1}y^{a+c-1}(1-y)^{d-1} \ \diff{y} \ \diff{z}= 1-t^{ak}\frac{\Beta(a+c,d)}{\Beta(c,d)},
\end{align*}
where for the inequality (1) we used that $t^kyz\leq z$ when $t\leq 1$.
If $t\geq 1$, then
\begin{align*}
\Prob\{(X/Y)^{1/k} \geq t\}  &= C \cdot \int_{x=0}^1 \int_{y=0}^{t^{-k}x} x^{a-1}(1-x)^{b-1}y^{c-1}(1-y)^{d-1} \ \diff{y} \ \diff{x}\\
&\stackrel{y=(t^{-k}x)z}{=} C \cdot \frac{1}{t^{ck}} \int_{x=0}^1  \int_{z=0}^1 x^{a+c-1}(1-x)^{b-1}z^{c-1}(1-t^{-k}xz)^{d-1} \ \diff{z} \ \diff{x}\\
&\stackrel{(2)}{\leq} C \cdot \frac{1}{t^{ck}} \int_{x=0}^1  \int_{z=0}^1 x^{a+c-1}(1-x)^{b-1}z^{c-1} \ \diff{z} \ \diff{x} =  \frac{1}{t^{ck}} \frac{\Beta(a+c,b)}{c\Beta(a,b)\Beta(c,d)},
\end{align*}
where for (2) we used that $t^{-k}xz\geq 0$. 
\end{proof}

\end{document}